\documentclass[11pt,reqno,a4paper]{amsart}
\usepackage[T1]{fontenc}
\usepackage{lmodern,amssymb,mathtools}
\usepackage{microtype,xcolor}
\usepackage[colorlinks=true,linkcolor=blue,citecolor=blue,
 urlcolor=blue,bookmarksnumbered=true]{hyperref}
\hypersetup{
 pdftitle={Singular Yamabe metrics on complements of closed sets},
 pdfauthor={Weiyang Li and Heng Zhang},
 pdfsubject={A construction by uniform finite point gluing},
 pdfkeywords={singular Yamabe problem, Delaunay ends, Hausdorff dimension,
 Newtonian potentials, gluing}}
\numberwithin{equation}{section}
\newtheorem{theorem}{Theorem}[section]
\newtheorem{proposition}[theorem]{Proposition}
\newtheorem{lemma}[theorem]{Lemma}
\newtheorem{question}[theorem]{Question}

\theoremstyle{definition}

\theoremstyle{remark}

\newcommand{\R}{\mathbb R}
\newcommand{\Sph}{\mathbb S}
\newcommand{\dd}{\mathrm d}
\newcommand{\Hh}{\mathcal H}

\newcommand{\cM}{\mathcal M}
\newcommand{\cE}{\mathcal E}
\newcommand{\cB}{\mathcal B}
\newcommand{\norm}[1]{\left\lVert#1\right\rVert}

\DeclareMathOperator{\dist}{dist}
\DeclareMathOperator{\diam}{diam}
\DeclareMathOperator{\supp}{supp}

\title[Singular Yamabe metrics]{Singular Yamabe metrics on complements of closed sets}
\author[W. Li and H. Zhang]{Weiyang Li and Heng Zhang}
\address{School of Mathematical Sciences, University of Science and Technology of China, Hefei, China}
\email[Weiyang Li]{liweiyang@ustc.edu.cn}
\address{School of Mathematical Sciences, University of Science and Technology of China, Hefei, China}
\email[Heng Zhang]{hengz@mail.ustc.edu.cn}
\keywords{Singular Yamabe problem, Hausdorff dimension, Newtonian potentials, Elliptic gluing}
\subjclass[2020]{Primary 53C21; Secondary 35J60, 35J75}
\date{}

\begin{document}
\begin{abstract}
Let $n\geq3$ and let $\Lambda\subset\Sph^n$ be closed, contain at least
two points, and satisfy $\dim_{\mathrm H}\Lambda<(n-2)/2$.
We prove that there exists a complete
metric of scalar curvature $n(n-1)$ on $\Sph^n\setminus\Lambda$,
conformal to the round metric. 
\end{abstract}
\maketitle
\section{Introduction}
\subsection{Question and main result}
The Yamabe problem asks whether every conformal class on a closed, connected
Riemannian manifold of dimension at least three contains a metric of
constant scalar curvature. Yamabe proposed a solution in 1960
\cite{Yamabe1960}. In 1968,
Trudinger found a gap in the compactness argument and
proved the result under a restriction on the conformal variational
constant \cite{Trudinger1968}. Eliasson gave another proof with a more
precise sufficient hypothesis \cite{Eliasson1971}. These arguments
settled an important range of conformal classes, but not the general
problem.

Aubin established the key comparison with the round sphere. A
minimizer exists when the Yamabe constant is strictly less than that
of the sphere, and he proved this strict inequality in dimension
$n\geq6$ when the manifold is not locally conformally flat
\cite{Aubin1976}. Schoen completed the argument in 1984
\cite{Schoen1984}. His method uses the Green function of the conformal
Laplacian and a global positive mass theorem, in contrast to Aubin's
local test functions. The positive mass and positive action results of Schoen and Yau provided essential foundation \cite{SY1979mass,SY1979action}; the locally conformally flat higher-dimensional case is discussed in \cite{SY1988}.
Thus every such conformal class has a constant scalar curvature
representative. Lee and Parker give a detailed account of the proof
and its history \cite{LeeParker1987}.

On a noncompact manifold, the corresponding geometric problem also
requires the conformal metric to be complete. This additional
requirement is essential: Jin constructed a complete metric whose
conformal class contains no complete metric of constant scalar
curvature \cite{Jin1988}. We study the positive singular Yamabe
problem on domains in the round sphere, where the conformal geometry
is fixed and the question is whether the removed set can be placed
at infinite metric distance while the scalar curvature is a positive
constant.

Write $g_{\mathrm{rd}}$ for the round metric of sectional curvature
one on $\Sph^n$, with $n\geq3$. Yau posed the following interesting question in the problem section of \emph{Seminar on Differential Geometry}, attributing it to Schoen \cite[Problem 32, p. 677]{Yau1982}.

\begin{question}\label{question:yau}
If $\Lambda\subset\Sph^n$ is closed, contains at least two points, and
satisfies
\begin{equation}\label{eq:dimension}
 \dim_{\mathrm H}\Lambda<\frac{n-2}{2},
\end{equation}
does $\Sph^n\setminus\Lambda$ admit a complete metric conformal to
$g_{\mathrm{rd}}$ with scalar curvature $n(n-1)$?
\end{question}

We give an affirmative answer, together with a global distributional
formulation of the solution. Set
\begin{equation}\label{eq:constants}
\begin{aligned}
 a&=\frac{n-2}{2},& p&=\frac{n+2}{n-2}=1+\frac2a,\\
 q&=p-1,& \kappa&=\frac{n(n-2)}4=a(a+1).
\end{aligned}
\end{equation}
We use the sign convention $\Delta=\sum_j\partial_j^2$ in Euclidean
coordinates. For $g=u^{2/a}g_{\mathrm{rd}}$, the equation
$R(g)=n(n-1)$ is
\begin{equation}\label{eq:round}
 \mathcal A_{\mathrm{rd}}u=\kappa u^p,\qquad \mathcal A_{\mathrm{rd}}=-\Delta_{g_{\mathrm{rd}}}+\kappa.
\end{equation}

\begin{theorem}\label{thm:main}
Let $n\geq3$, and let $\Lambda\subset\Sph^n$ be closed, contain at
least two points, and satisfy \eqref{eq:dimension}. Then
$\Omega=\Sph^n\setminus\Lambda$ admits a complete metric
$g=u^{2/a}g_{\mathrm{rd}}$ of scalar curvature $n(n-1)$, with
$u\in C^\infty(\Omega)$ positive. Moreover, $u\in L^p(\Sph^n)$ and
$\mathcal A_{\mathrm{rd}}u=\kappa u^p$ on the whole sphere in the sense of distributions.
\end{theorem}

\subsection{The dimension threshold and previous results}

The two restrictions on $\Lambda$ have different origins. The
exclusion of a one-point singular set is necessary: stereographic
projection would otherwise produce a positive entire solution of the
critical Euclidean equation. Its classification gives a standard
bubble, whose conformal metric is incomplete at infinity
\cite{CGS1989}. For the empty singular set, the round metric provides
a solution.

The dimension restriction is suggested by the flat model for the
complement of a smooth $k$-dimensional singular set. In coordinates
$(y,r,\omega)\in\R^k\times(0,\infty)\times\Sph^{n-k-1}$, the metric
\begin{equation}\label{eq:intro-product}
 r^{-2}\bigl(|\dd y|^2+\dd r^2+r^2g_{\Sph^{n-k-1}}\bigr)
 =\frac{|\dd y|^2+\dd r^2}{r^2}+g_{\Sph^{n-k-1}}
\end{equation}
is the product of hyperbolic $(k+1)$-space and the unit
$(n-k-1)$-sphere. It dominates $\dd r^2/r^2$, so $r=0$ lies at
infinite distance. Its scalar curvature is
\begin{equation}\label{eq:intro-scalar}
 -k(k+1)+(n-k-1)(n-k-2)=(n-1)(n-2-2k),
\end{equation}
which is positive when $k<(n-2)/2$. This model is the basis for the smooth singular set constructions of Mazzeo--Smale and Mazzeo--Pacard \cite{MazzeoSmale1991,MP1996}.

The same threshold occurs in a necessary condition valid without
smoothness assumptions on the boundary. Schoen and Yau
\cite[Theorem~2.7]{SY1988} proved that a domain
$\Omega\subset\Sph^n$ admitting a complete metric conformal to the
round metric, with scalar curvature bounded below by a positive
constant, must satisfy
\begin{equation}\label{eq:intro-necessary}
 \dim_{\mathrm H}\partial\Omega\leq\frac{n-2}{2}.
\end{equation}
For nonnegative scalar curvature, the estimate in the original
theorem of Schoen and Yau is
$\dim_{\mathrm H}\partial\Omega\leq n/2$.
For subsequent refinements, see
\cite[Theorem~C]{Carron2012} and
\cite[Theorem~1.3]{MaQing2025}.
In particular, Ma and Qing proved that a conformal metric
with nonnegative scalar curvature, geodesically complete
near a compact singular set $S$ in a smooth Riemannian
background, satisfies
$\dim_{\mathrm H}S\leq(n-2)/2$.
Since $\Lambda$ is closed and has Hausdorff dimension less than $n$,
it has empty interior and $\partial\Omega=\Lambda$.
Theorem~\ref{thm:main} shows that the strict Hausdorff dimension
bound $\dim_{\mathrm H}\Lambda<(n-2)/2$ is sufficient for
existence, provided that $\Lambda$ contains at least two points.

The equality case is not covered by our theorem. Although the model
\eqref{eq:intro-product} is scalar flat when $k=(n-2)/2$, complete
positive scalar curvature metrics can still exist at this dimension.
For even $n\geq4$, Pacard constructed such metrics for finite disjoint
unions of compact smooth submanifolds without boundary of dimension
$(n-2)/2$ \cite{Pacard1995}. Chan and DelaTorre later gave an analytic construction of
local singular Yamabe metrics near smooth submanifolds of
dimension $(n-2)/2$ \cite{ChanDelaTorre2020}.

The corresponding existence theories for other signs of curvature
have a different form. The negative case originates in
Loewner--Nirenberg \cite{LoewnerNirenberg1974}; on compact backgrounds,
the smooth set dimension range is $k>(n-2)/2$
\cite{AvilesMcOwen1988,KatoNayatani1993}. For general closed sets,
Labutin's criterion is a Wiener-type capacity condition at each
boundary point \cite{Labutin2005}. Scalar flat metrics on smooth
complements were constructed in
\cite{MaMcOwen1992,KatoNayatani1993}. For general sphere domains,
Karakhanyan \cite[Theorem~1.1]{Karakhanyan2024} characterized
scalar flat existence by $\mathcal C_{1+2/n,n/2}(\Lambda)=0$, where
$\mathcal C$ denotes Bessel capacity. These criteria concern
different equations and do not yield a positive curvature existence
theorem under \eqref{eq:dimension}.

Earlier existence results for metrics of constant positive scalar
curvature cover several classes of singular sets of Hausdorff
dimension less than $(n-2)/2$. Schoen constructed such metrics
for every finite set of at least two points, as well as certain
infinite singular sets \cite{Schoen1988}. Connected-sum
constructions give further infinite examples under specific
geometric arrangement hypotheses \cite{MPU1995}. For smooth
singular sets in this dimension range, Mazzeo and Smale treated
perturbations of suitable spheres \cite{MazzeoSmale1991}, and
Mazzeo and Pacard constructed metrics for arbitrary
positive-dimensional components \cite{MP1996}. Their finite point
and mixed dimensional constructions appear in \cite{MP1999}.
Related constructions with singular sets given by smooth
submanifolds with boundary were obtained by Fakhi
\cite{Fakhi2003}.
Within the dimension range \eqref{eq:dimension},
Theorem~\ref{thm:main} gives existence for every closed set
containing at least two points, without smoothness or
geometric arrangement assumptions.

A positive global weak solution of the Yamabe equation need not
induce a complete metric on the complement of its singular set.
In dimensions $n=4$ and $6$, Pacard constructed positive weak
solutions whose singular sets have Hausdorff dimension greater
than $(n-2)/2$ \cite{Pacard1994}. The corresponding conformal
metrics are necessarily incomplete by \eqref{eq:intro-necessary}.
For related constructions of weak solutions with prescribed
singular sets, see \cite{ChenLin1999}. Theorem~\ref{thm:main}
gives a complete metric on $\Sph^n\setminus\Lambda$ whose conformal
factor also satisfies \eqref{eq:round} on the whole sphere in
the sense of distributions.

\subsection{Proof strategy}

We start with a noncylindrical Delaunay metric, that is a complete
rotationally symmetric model of scalar curvature $n(n-1)$ on the
sphere minus two points. A conformal transformation places these
punctures at two chosen points of $\Lambda$. We then insert finite
families of new singular points in $\Lambda$, obtaining complete
metrics $g_j=u_j^{2/a}g_{\mathrm{rd}}$ at each step. Finitely many
insertions suffice when $\Lambda$ is finite. For infinite $\Lambda$,
we must control the limit as well as the individual steps.

The dimension hypothesis enters through the choice of covers.
Choose $\dim_{\mathrm H}\Lambda<\beta<a$, so that
$\Hh^\beta(\Lambda)=0$, where $\Hh^\beta$ denotes $\beta$-dimensional
Hausdorff measure. Outside small neighborhoods of the existing
punctures, the remaining part of $\Lambda$ is compact. In a fixed
stereographic chart, Lemma~\ref{pot:cover} provides disjoint
Euclidean balls $B(y_i,r_i)$ centered on this remainder, whose
fivefold enlargements $B(y_i,5r_i)$ cover it. Both the maximal radius $\max_i r_i$ and the sum
$\sum_i r_i^\beta$ can be made arbitrarily small.
They satisfy the following uniform $\beta$-packing bound:
$$
 \sum_{\substack{y_i\in B(x,R)\\ r_i\le R}} r_i^\beta
 \le C R^\beta
 \qquad (x\in\R^n,\ R>0),
$$
where $C$ depends only on $n$ and $\beta$.
This bound is a property of the chosen family of balls.

For a fixed $\eta>0$, set $\lambda_i=\eta r_i$.
The standard bubble
$\left(2\lambda_i/(\lambda_i^2+|x-y_i|^2)\right)^a$ has leading term
$(2\lambda_i)^a|x-y_i|^{2-n}$ far from its center.
The sum $Q$ of these Newtonian coefficients therefore satisfies
\begin{equation}\label{pot:charge}
 Q:=\sum_i(2\eta r_i)^a
 \leq (2\eta)^a\bigl(\max_i r_i\bigr)^{a-\beta}
                \sum_i r_i^\beta.
\end{equation}
Thus the radii and $Q$ can be made arbitrarily small while $\eta$
and the packing constant remain fixed.

The key analytic tool is Theorem~\ref{thm:insertion}. Under a nondegeneracy hypothesis
on the current finite step metric, it inserts complete Delaunay
ends for these packed families and allows arbitrarily small perturbations on
compact sets separated from the insertion region. The hypothesis,
called marked nondegeneracy, excludes nonzero infinitesimal
conformal deformations solving the linearized equation and
decaying exponentially at every end. The initial metric satisfies
this condition, and gluing preserves it, so each produced metric is an
admissible background for the next insertion.

The crucial point is that the smallness threshold for $\eta$
is independent of the current background. Once $\eta$ is fixed,
the constants in the lower bounds are also independent of the
step, the number of new points, and their radius ratios.
The admissible upper bounds on the radii and $Q$ may instead
depend on the background and the approximation accuracy.
The uniform estimates combine local control of the Delaunay ends
with bounds on interaction errors relative to positive Newtonian
potentials. We can therefore keep $\eta$ fixed and choose summable
perturbations in local smooth norms, with pointwise errors also
summable relative to the preceding conformal factor. These give
a positive smooth limit $u$ satisfying
\eqref{eq:round} on $\Omega$.

Completeness requires more than local smooth convergence.
At step $j$, we construct a Lipschitz cutoff $0\leq f_j\leq1$
that equals one near $\Lambda$ and is supported in a neighborhood
shrinking to $\Lambda$. The transition regions
$W_j=\{1/4\leq f_j\leq3/4\}\Subset\Omega$ are pairwise disjoint.
The lower bounds from the gluing construction give the metric
gradient estimate $|\dd f_j|_{g_j}\leq C/\eta$ almost everywhere
on $W_j$. Summable relative perturbations on the earlier regions
preserve this bound, up to a fixed factor, for
$g=u^{2/a}g_{\mathrm{rd}}$. Any curve that eventually leaves every
compact subset of $\Omega$ starts outside $\supp f_j$ for all
sufficiently large $j$ and eventually reaches the region where
$f_j=1$. It must therefore cross each such $W_j$ from level $1/4$
to level $3/4$. Each crossing has $g$-length at least $c\eta$,
with $c>0$ independent of $j$. The disjointness of the regions
then forces the curve to have infinite length.

Finally, each $u_j$ belongs to $L^p(\Sph^n)$ and already satisfies
\eqref{eq:round} globally in distributions.
Proposition~\ref{prop:no-defect} gives the uniform estimate
$$
 \int_{B_{g_{\mathrm{rd}}}(x,r)}u_j^p\,\dd V_{g_{\mathrm{rd}}}
 \leq C_n r^a,\qquad 0<r\leq r_0,
$$
for round balls, with $C_n,r_0>0$ independent of $j$ and $x$. Since
$\Hh^a(\Lambda)=0$, covers of $\Lambda$ with arbitrarily small
$\sum_i r_i^a$ make the integrals of $u_j^p$ uniformly small
on a neighborhood of $\Lambda$. Together with local smooth
convergence away from $\Lambda$, this gives
$u_j^p\to u^p$ in $L^1(\Sph^n)$ and $u_j\to u$ in
$L^p(\Sph^n)$. We can therefore pass to the global distributional
equation.

\subsection{Organization of paper}

Section~\ref{sec:finite-gluing} introduces the uniform finite point gluing theorem,
and establishes the covering and Newtonian potential estimates.
Section~\ref{in:section} constructs the local Delaunay families,
and Section~\ref{old:section} develops the boundary operators
and inverse estimates for the old ends.
Section~\ref{sec:nonlinear-matching} solves the nonlinear exterior
problem and matches the interior and exterior families, proving
Theorem~\ref{thm:insertion} and preservation of marked
nondegeneracy. Section~\ref{sec:limit} carries out the iteration
and proves Theorem~\ref{thm:main}, including completeness and
the global distributional equation.

\section{Finite point gluing and potential estimates}
\label{sec:finite-gluing}

\subsection{Finite-ended solutions and nondegeneracy}
\label{subsec:finite-backgrounds}

We formulate the gluing result in stereographic coordinates.
Write $g_{\mathrm{rd}}=\vartheta^{2/a}|\dd x|^2$, where
$\vartheta(x)=(2/(1+|x|^2))^a$. The flat conformal factor
$U=\vartheta u$ then satisfies
\begin{equation}\label{eq:flat}
 \Delta U+\kappa U^p=0,\qquad
 g=U^{2/a}|\dd x|^2.
\end{equation}

Let $F\subset\Sph^n$ be finite. We consider positive solutions
$u$ of \eqref{eq:round} on $\Sph^n\setminus F$ such that the
conformal metric $g=u^{2/a}g_{\mathrm{rd}}$ is complete and is
asymptotic to a noncylindrical Delaunay metric at each puncture.
For brevity, we call such a solution a \emph{finite-ended solution}.
Each puncture corresponds to one end of $(\Sph^n\setminus F,g)$.
Near a puncture taken as the origin, write $r=|x|$ and
$t=-\log r$. The Delaunay condition means that
$$
 r^aU(r,\omega)=v_\varepsilon(t+\tau)+R(t,\omega),
$$
where $v_\varepsilon$ is a positive nonconstant periodic
solution of $v''-a^2v+\kappa v^p=0$, with minimum
$\varepsilon>0$, which we call the neck parameter.
The parameter $\tau$ denotes the phase. The remainder
$R$ and its cylindrical derivatives of order at most two
decay exponentially as $t\to\infty$
\cite{MP1999,MPU1996}.

Such a solution is called \emph{marked nondegenerate} if
the only global solution of
$$
 (\Delta_g+n)j=0
$$
that decays exponentially at every end is zero
\cite[Definition~3]{MP1999}. We impose this condition only
at finite steps.

\subsection{Uniform finite point gluing}
\label{subsec:gluing-statement}

Fix a packing exponent $\beta\in(0,a)$ and a constant
$C_{\mathrm{pack}}>0$. Also fix the auxiliary H\"older
exponent $\theta\in(0,1)$ and weights
$1<\mu_1<\mu_2<\mu_3<2$ used in Section~\ref{in:section}.
Constants depending only on
$n,\beta,C_{\mathrm{pack}},\theta,\mu_1,\mu_2,\mu_3$
will be called universal.

Let $u_0$ be a marked nondegenerate finite-ended solution
with singular set $F$, and set $U_0=\vartheta u_0$.
We refer to $u_0$, or to $U_0$ in flat coordinates, as the
background solution for the current gluing step. Choose a stereographic
pole outside $F$ and a compact core $D$, obtained by
truncating the old ends and removing a small regular cap
about the pole. Fix a compact set $K\Subset D$ of permitted
new centers. Consider finite families of pairwise disjoint
balls $B(y_i,r_i)$, with $y_i\in K$, satisfying
\begin{equation}\label{eq:packing}
 \sum_{\substack{y_i\in B(x,R)\\r_i\leq R}}r_i^\beta
 \leq C_{\mathrm{pack}}R^\beta
 \qquad(x\in\R^n,\ R>0).
\end{equation}
For $\eta>0$, set
\begin{equation}\label{eq:scales}
 \lambda_i=\eta r_i,\qquad
 \rho_i=\eta^{2/(n+2)}r_i,\qquad
 A_i=(2\lambda_i)^a,\qquad
 Q=\sum_i A_i.
\end{equation}
Here $\lambda_i$ is the bubble scale and $\rho_i$ is the
matching radius. The coefficient $A_i$ is the coefficient
of $|x|^{2-n}$ in the tail of the standard bubble
$b_{\lambda_i}(x)
=(2\lambda_i/(\lambda_i^2+|x|^2))^a$.

\begin{theorem}[Uniform finite point gluing]\label{thm:insertion}
There exists a universal $0<\eta_*\leq1$ with the following
property. For each fixed $0<\eta\leq\eta_*$ and each choice of data $(U_0,D,K)$ as above, every such family with sufficiently small
$\max_i r_i$ and $Q$ admits a marked nondegenerate finite-ended
solution $U_1$ with singular set $F\cup\{y_i\}$.
On $D\setminus\bigcup_i B(y_i,\rho_i)$,
\begin{equation}\label{eq:core-lower}
 U_1\geq c\left(
 U_0+\sum_i\frac{A_i}{(\rho_i^2+|x-y_i|^2)^a}
 \right),
\end{equation}
where $c>0$ depends only on the universal data and $\eta$.
By decreasing the radius and coefficient bounds, $u_1$ can be
made arbitrarily close to $u_0$ in any prescribed finite
collection of $C^m$ norms on compact subsets of
$\Sph^n\setminus(F\cup K)$, while the old neck and phase
parameters remain in any prescribed neighborhood of their
original values.
\end{theorem}

For an empty insertion family take $U_1=U_0$; the nontrivial
gluing construction below assumes $Q>0$.

The radius bound can also ensure $B(y_i,10r_i)\Subset D$.
The individual terms in \eqref{eq:core-lower} then give
\begin{equation}\label{eq:shell-lower}
 |x-y_i|\leq10r_i
 \quad\Longrightarrow\quad
 U_1(x)\geq c\eta^a r_i^{-a}
\end{equation}
outside $\bigcup_j B(y_j,5r_j)$, after decreasing $c$
if necessary.

The proof combines Silva Santos's local inverse
\cite{Santos2010} with the Fredholm theory for the old ends
\cite{MPU1996}, within the gluing framework of
\cite{Mazzeo1991,MPU1995,MP1999,Byde2003,Ratzkin2003}.
The uniform estimates control the period drift caused by
neck variation, using a loss of weight, and the interactions
between new singularities, using relative Newtonian potential
bounds. Remainders dependent of background are made small by
reducing the radii and $Q$ after $\eta$ is fixed.

The finite point solutions in \cite{MP1999} are unmarked
nondegenerate, with marked nondegeneracy established for
generic configurations. Here preservation of marked
nondegeneracy is proved in Subsection~\ref{sec:matching}.

\subsection{Covers and Newtonian potentials}\label{pot:section}

Throughout this section, $0<\beta<a$, with $a$, $p$, and $\kappa$
as in \eqref{eq:constants}.  We use the normalized Newtonian
potential
$$
 \mathcal I_2\mu(x)=\gamma_n\int_{\R^n}|x-z|^{2-n}\,\dd\mu(z),
 \qquad \gamma_n=\frac1{(n-2)|\Sph^{n-1}|}.
$$
For a nonnegative density $f$, we write $\mathcal I_2f$ for
$\mathcal I_2(f\,\dd x)$; whenever the potential is locally
integrable, $-\Delta\mathcal I_2f=f$ in distributions.
Background on coverings and Hausdorff measure can be found in
\cite[Chapters~2 and~4]{Mattila1995}, and on potentials and capacity
in \cite[Chapters~2--3]{AdamsHedberg1996}.  Unless specified
otherwise, the constants in the packing and potential estimates
depend only on $n$, $\beta$, and the packing constant.  

\subsubsection{Finite covers with a packing bound}

The following cover result is a finite version of the dyadic
construction in \cite[Lemma~2.2]{Netrusov1994}.

\begin{lemma}\label{pot:cover}
Let $E\subset\R^n$ be compact, with $\Hh^\beta(E)=0$, and let
$\varepsilon,d>0$.  There is a finite family of balls
$B(y_i,r_i)$, with $y_i\in E$, such that:
\begin{enumerate}
\item their closures are pairwise disjoint;
\item $E\subset\bigcup_i B(y_i,5r_i)$;
\item $\max_i r_i<d$ and $\sum_i r_i^\beta<\varepsilon$;
\item for every $x\in\R^n$ and $R>0$,
\begin{equation}\label{pot:packing}
 \sum_{\substack{y_i\in B(x,R)\\r_i\le R}}r_i^\beta
 \le C(n,\beta)R^\beta.
\end{equation}
\end{enumerate}
If $E\subset U$ for an open set $U$, the family can also be chosen
so that $B(y_i,10r_i)\subset U$ for every $i$.
\end{lemma}

\subsubsection{Bubbles and regularized Newtonian potentials}

The normalized bubble centered at $y$ with scale $\lambda$ is
\begin{equation}\label{pot:bubble}
 b_{\lambda,y}(x)
 =\left(\frac{2\lambda}{\lambda^2+|x-y|^2}\right)^a.
\end{equation}
It satisfies $\Delta b_{\lambda,y}+\kappa b_{\lambda,y}^p=0$,
and its asymptotic coefficient is $(2\lambda)^a$:
$$
 b_{\lambda,y}(x)
 =(2\lambda)^a|x-y|^{-2a}
   \bigl(1+O(\lambda^2|x-y|^{-2})\bigr)
 \qquad (|x-y|\to\infty).
$$
Let the pairwise disjoint balls $B(y_i,r_i)$ satisfy
\eqref{pot:packing} with constant $C_{\mathrm{pack}}$.
In the potential estimates below, $\nu$ is an independent parameter;
in the gluing application we take $\nu=\eta^{2/(n+2)}$.
For $0<\eta,\nu\le1$, set
$$
 \lambda_i=\eta r_i,\qquad \rho_i=\nu r_i,\qquad
 A_i=(2\lambda_i)^a,\qquad \delta=(\eta/\nu)^2,
$$
and define
\begin{equation}\label{pot:tails}
 h_i(x)=\frac{A_i}{(\rho_i^2+|x-y_i|^2)^a},
 \qquad h=\sum_i h_i,\qquad H_i=\sum_{j\ne i}h_j.
\end{equation}
These regularized Newtonian potentials obey
$h_i=(\lambda_i/\rho_i)^a b_{\rho_i,y_i}$.

\subsubsection{A nonlinear potential estimate}

The next estimate is a quantitative form, for the present
packing family, of the nonlinear potential condition in
\cite[Theorem~1.2]{KaltonVerbitsky1999}.  We obtain its uniform
$O(\delta)$ constant directly.

\begin{proposition}\label{pot:endpoint}
The potentials in \eqref{pot:tails} satisfy
\begin{equation}\label{pot:endpointineq}
 \mathcal I_2(h^p)\le C\delta h
 \qquad\text{on }\R^n,
\end{equation}
and
\begin{equation}\label{pot:mass}
 \int_{\R^n}h^p\,\dd x\le C\delta Q.
\end{equation}
\end{proposition}

\begin{proof}
Set $\gamma=(a-\beta)/a\in(0,1)$ and
$D_j(x)=(\rho_j^2+|x-y_j|^2)^{1/2}$.  We first prove
\begin{equation}\label{pot:pointwise}
 h(x)^p\le C\eta^{a+2}\nu^{\gamma-2}
       \sum_j r_j^{a+\gamma}D_j(x)^{-n-\gamma}.
\end{equation}
Suppose $x$ is outside every original ball, so that
$d_j=|x-y_j|\ge r_j$, and write
$$
 S=\sum_j r_j^ad_j^{-2a},\qquad
 T=\sum_j r_j^{a+\gamma}d_j^{-n-\gamma}.
$$
Since $n=2a+2$ and $a\gamma+\beta=a$,
$$
 r_j^ad_j^{-2a}
 =\bigl(r_j^{a+\gamma}d_j^{-n-\gamma}\bigr)^{a/(a+1)}
   \bigl(r_j^\beta d_j^{a-\beta}\bigr)^{1/(a+1)}.
$$
H\"older's inequality and packing give, for $R>0$,
$$
 \sum_{d_j\le R}r_j^ad_j^{-2a}
 \le T^{a/(a+1)}
      \left(\sum_{d_j\le R}r_j^\beta d_j^{a-\beta}\right)^{1/(a+1)}
 \le C(RT)^{a/(a+1)}.
$$
On a shell $2^kR<d_j\le2^{k+1}R$, all $r_j$ are at most
$2^{k+1}R$, and hence
$$
 \sum r_j^a\le(2^{k+1}R)^{a-\beta}\sum r_j^\beta
 \le C(2^kR)^a.
$$
The near--far decomposition, as in \cite{Hedberg1972}, therefore gives
$$
 S\le C(RT)^{a/(a+1)}+CR^{-a}.
$$
If the family is nonempty, take $R=T^{-1/(a+2)}$ to obtain
$S^p\le CT$.  Since $d_j\le D_j\le\sqrt2d_j$ and
$\nu^{\gamma-2}\ge1$, this proves \eqref{pot:pointwise} outside
the original balls.  The empty family is immediate.

If $x\in B(y_i,r_i)$, disjointness permits the same argument
for $H_i$.  For the remaining term, $D_i\ge\nu r_i$ and
$\gamma<2$ imply
$$
 h_i^p
 =2^{a+2}\eta^{a+2}r_i^{a+2}D_i^{-n-2}
 \le C\eta^{a+2}\nu^{\gamma-2}
          r_i^{a+\gamma}D_i^{-n-\gamma}.
$$
Together with $(h_i+H_i)^p\le2^{p-1}(h_i^p+H_i^p)$, this
proves \eqref{pot:pointwise} everywhere.

For $f_\rho(t)=(\rho^2+t^2)^{-(n+\gamma)/2}$, the radial
Newtonian formula bounds $\mathcal I_2f_\rho(D)$ by
$C\rho^{2-n-\gamma}$ for $D\le\rho$, and by
$C\rho^{-\gamma}D^{2-n}$ for $D\ge\rho$. Thus
\begin{equation}\label{pot:convolution}
 \mathcal I_2\bigl((\rho^2+|\cdot-y|^2)^{-(n+\gamma)/2}\bigr)(x)
 \le C\rho^{-\gamma}(\rho^2+|x-y|^2)^{-a}.
\end{equation}
Applying this to \eqref{pot:pointwise}, using $\rho_j=\nu r_j$,
yields
$$
 \mathcal I_2(h^p)
 \le C\eta^{a+2}\nu^{-2}
       \sum_jr_j^a(\rho_j^2+|x-y_j|^2)^{-a}
 =C\delta h.
$$
Multiplying \eqref{pot:endpointineq} by $|x|^{n-2}$,
letting $x$ tend to infinity along a ray, and applying Fatou's
lemma gives $\gamma_n\int h^p\le C\delta Q$, since
$|x|^{n-2}h(x)\to Q$. This proves \eqref{pot:mass}, including
integrability.
\end{proof}

\subsubsection{Relative estimates for positive potentials}

\begin{lemma}\label{pot:positive}
Each $h_i$ is the Newtonian potential of a positive finite measure.
If $V\ge0$ satisfies $\sup_x\mathcal I_2V(x)\le L$ and
$W=\mathcal I_2\mu$ for a positive measure $\mu$, then
\begin{equation}\label{pot:3Goperator}
 \mathcal I_2(VW)(x)\le C_nLW(x)
\end{equation}
wherever $W(x)$ is finite.
\end{lemma}

\begin{proof}
Using the identity
$h_i=(\lambda_i/\rho_i)^a b_{\rho_i,y_i}$,
the bubble equation, and decay at infinity, we obtain
\begin{equation}\label{pot:taildensity}
 -\Delta h_i=\kappa\delta^{-1}h_i^p>0,
 \qquad h_i=\kappa\delta^{-1}\mathcal I_2(h_i^p),
\end{equation}
with $h_i^p\in L^1(\R^n)$.
For $K(x,z)=\gamma_n|x-z|^{2-n}$, the standard three-kernel
inequality (see \cite[equations~(1.6) and~(3.1)]{KaltonVerbitsky1999}) is
\begin{equation}\label{pot:3G}
 K(x,z)K(z,y)
 \le C_nK(x,y)\bigl(K(x,z)+K(z,y)\bigr).
\end{equation}
Integrating against $V(z)\,\dd z\,\dd\mu(y)$ and using Tonelli's
theorem gives
$$
 \mathcal I_2(VW)(x)
 \le C_n\int K(x,y)
     \bigl(\mathcal I_2V(x)+\mathcal I_2V(y)\bigr)\,\dd\mu(y)
 \le C_nLW(x).
$$
\end{proof}

\begin{proposition}\label{pot:relative}
For each index $i$ and any positive Newtonian potential $W$,
\begin{equation}\label{pot:singlemultiplier}
 \mathcal I_2(h_i^{p-1}W)\le C_n\delta W.
\end{equation}
\end{proposition}

\begin{proof}
The radial Newtonian formula gives
\begin{equation}\label{pot:singleKato}
 \sup_x\mathcal I_2(h_i^{p-1})(x)
 =\frac{4\eta^2r_i^2}{n-2}
       \int_0^\infty\frac{t\,\dd t}{(\rho_i^2+t^2)^2}
 =\frac\delta a.
\end{equation}
The integral converges even when $h_i^{p-1}$ has infinite total
mass. Lemma~\ref{pot:positive} now proves
\eqref{pot:singlemultiplier}.
\end{proof}

For $\eta\le\nu$, the true bubbles satisfy
$h_j\le b_{\lambda_j,y_j}\le2^ah_j$ wherever
$|x-y_j|\ge\rho_j$.

\subsubsection{Distributional and Newtonian representations}

\begin{lemma}\label{pot:oldrepresentation}
Let $u_0$ be the positive round conformal factor of a finite-ended
solution with puncture set $F$ and Delaunay asymptotics.  Then
$u_0\in L^p(\Sph^n)$ and
\begin{equation}\label{eq:finite-distribution}
 \mathcal A_{\mathrm{rd}}u_0=\kappa u_0^p\qquad\text{in }\mathcal D'(\Sph^n).
\end{equation}
In particular, the equation has no Dirac mass at a puncture.
If the stereographic pole lies outside $F$, the flat conformal
factor $U_0=\vartheta u_0$ satisfies $U_0^p\in L^1(\R^n)$ and
\begin{equation}\label{pot:oldnewton}
 U_0=\kappa\mathcal I_2(U_0^p)
 \qquad\text{on }\R^n\setminus F.
\end{equation}
Consequently,
\begin{equation}\label{pot:oldmultiplier}
 \mathcal I_2(h_i^{p-1}U_0)\le C_n\delta U_0,
\end{equation}
with a constant independent of $U_0$.
\end{lemma}

\begin{proof}
In a flat chart about each puncture, the Euclidean conformal factor
is a nonnegative $C^2$ solution of $-\Delta U=\kappa U^p$.
Since $p=(n+2)/(n-2)>n/(n-2)$, \cite[Lemma~2.1]{CGS1989}
gives local $L^p$ integrability and the distributional equation
across the puncture.  Conformal covariance therefore yields
$u_0\in L^p(\Sph^n)$ and \eqref{eq:finite-distribution}.

The stereographic pole is regular, so $U_0=O(|x|^{2-n})$ and
$U_0^p=O(|x|^{-n-2})$ at infinity.  Thus $U_0^p\in L^1(\R^n)$,
and the standard near--far estimate gives
$\mathcal I_2(U_0^p)=O(|x|^{2-n})$.  Conformal covariance also
gives $-\Delta U_0=\kappa U_0^p$ in $\mathcal D'(\R^n)$.
Hence $U_0-\kappa\mathcal I_2(U_0^p)$ is distributionally
harmonic on all of $\R^n$ and tends to zero at infinity.
Weyl's lemma and Liouville's theorem prove \eqref{pot:oldnewton};
\eqref{pot:singlemultiplier} then gives \eqref{pot:oldmultiplier}.
\end{proof}

\subsubsection{Interaction and local amplitude estimates}

\begin{lemma}\label{pot:background}
Packing and disjointness imply
\begin{equation}\label{pot:interactionrow}
 \sup_i\sum_{j\ne i}
       \left(\frac{r_ir_j}{|y_i-y_j|^2}\right)^a\le C.
\end{equation}
Suppose the centers lie in a fixed compact subset of the regular
domain of $U_0$, and write
$$
 B_i=U_0(y_i)+\sum_{j\ne i}b_{\lambda_j,y_j}(y_i).
$$
For each finite derivative order $m_0$, the maximal radius can
be chosen sufficiently small, depending on $U_0$, the compact
set, and $\eta$, so that
\begin{align}
 B_ir_i^a&\le C\eta^a,\label{pot:backgroundsize}\\
 r_i^m\left|D^m\left(U_0+\sum_{j\ne i}b_{\lambda_j,y_j}\right)(y_i)\right|
 &\le C_mB_i\quad(0\le m\le m_0),
 \label{pot:backgroundjets}\\
 A_iB_i&\le C\eta^{2a},\qquad
 \frac{B_i\rho_i^{2a}}{A_i}\le C\nu^{2a}.
 \label{pot:backgroundconsequences}
\end{align}
The constants are independent of the old solution.
\end{lemma}

\begin{proof}
Disjointness gives $d_{ij}=|y_i-y_j|\ge r_i+r_j$.
The far-field shell estimate in the proof of
Proposition~\ref{pot:endpoint}, applied at $y_i$ with $R=r_i$
and the $i$th ball omitted, proves \eqref{pot:interactionrow}.
Hence the other bubbles contribute at most $C\eta^ar_i^{-a}$
to $B_i$.
Choose the maximal radius so that
$\sup U_0\,r_i^a\le\eta^a$ on the compact set.  This gives
\eqref{pot:backgroundsize}.

For $j\ne i$,
$$
 |D^mb_{\lambda_j,y_j}(y_i)|
 \le C_mb_{\lambda_j,y_j}(y_i)d_{ij}^{-m}.
$$
Since $r_i\le d_{ij}$, summing gives the bubble contribution
to \eqref{pot:backgroundjets}.  The old factor has bounded
relative derivatives on the fixed compact neighborhood, so
reducing the maximal radius gives the same estimate for it.
Finally, insert $A_i=2^a\eta^ar_i^a$ and $\rho_i=\nu r_i$
into \eqref{pot:backgroundsize} to obtain
\eqref{pot:backgroundconsequences}.
\end{proof}

\section{Local families of Delaunay ends}\label{in:section}

We construct the local solutions used at the new punctures for
\eqref{eq:flat}, with the constants in \eqref{eq:constants}.
We write $\Pi_{\mathrm{lo}}$ for the projection onto spherical harmonics
of degrees zero and one, and $\Pi_{\mathrm{hi}}$ for its complement.
Boundary norms are taken on $\Sph^{n-1}$ after radial rescaling.

Fix a H\"older exponent $0<\theta<1$ and weights
\begin{equation}\label{in:weights}
 1<\mu_1<\mu_2<\mu_3<2.
\end{equation}
For $\mathcal A_r=\{r/2<|x|<r\}$, set
$$
 |w|_{k,\theta;r}
 =\sum_{j=0}^k r^j\|\nabla^jw\|_{L^\infty(\mathcal A_r)}
   +r^{k+\theta}[\nabla^kw]_{\theta;\mathcal A_r}.
$$
The normalized weighted H\"older norms on $B_\rho\setminus\{0\}$ are
\begin{equation}\label{in:norms}
 \|w\|_{X_\mu}
 =\sup_{0<r\leq\rho}(r/\rho)^{-\mu}|w|_{2,\theta;r},
 \qquad
 \|f\|_{Y_\mu}
 =\sup_{0<r\leq\rho}(r/\rho)^{-\mu}r^2|f|_{0,\theta;r}.
\end{equation}
Overlapping annuli are used at the outer boundary. Set
$$
 X_\mu^{\mathrm{hi},0}
 =\{w\in X_\mu:\Pi_{\mathrm{hi}}w|_{\partial B_\rho}=0\}.
$$
This is the domain of each linearized operator below, equivalently the
range of its inverse. Both norms are $\rho^\mu$ times the corresponding
norms on $C^{2,\theta}_\mu$ and $C^{0,\theta}_{\mu-2}$ in
\cite[Section~2.4]{Santos2010}. If $\gamma\geq\mu$, the inclusions
$X_\gamma\subset X_\mu$ and $Y_\gamma\subset Y_\mu$ have norm at most one.

\subsection{A radial family with variable neck parameter}

Recall the bubble $b_\lambda=b_{\lambda,0}$ from
\eqref{pot:bubble}, and set $A=(2\lambda)^a$.
We hold the bubble coefficient $A$ fixed and, for $\rho>\lambda$, solve
the radial Cauchy problem
\begin{equation}\label{in:cauchy}
 U_B(\rho)=b_\lambda(\rho)+B,
 \qquad U_B'(\rho)=b_\lambda'(\rho),\qquad B>0.
\end{equation}
Put
$$
 s=\lambda/\rho,\qquad
 \vartheta_B=B\rho^{2a}/A,\qquad
 t=\log(\rho/r),\qquad v_B(t)=r^aU_B(r).
$$
We assume $s$, $\vartheta_B$, and $AB$ are small. Choose $B_*>0$ and
parametrize $B=B_*e^b$, with $|b|\leq b_0$ for a fixed small $b_0$.
Thus $b$ changes the neck parameter. The high-mode boundary datum
$\phi=B_*\psi$ will be independent of $b$.

The cylindrical equation and its conserved energy are
\begin{equation}\label{in:ode-energy}
 v''-a^2v+\kappa v^p=0,
 \qquad
 E(v,v')=\frac12\bigl((v')^2-a^2v^2+a^2v^{p+1}\bigr).
\end{equation}
At $B=0$ the solution is the standard bubble and $E=0$. Differentiating
the data in \eqref{in:cauchy} with respect to $b$ gives
\begin{equation}\label{in:energy-derivative}
 E_b=aB\rho^{2a+1}b_\lambda'(\rho)
       +\kappa B\rho^n\bigl(b_\lambda(\rho)+B\bigr)^p
     =-2a^2AB\bigl(1+O(s^2+\vartheta_B)\bigr).
\end{equation}
Indeed,
$$
 b_\lambda(\rho)=A\rho^{-2a}(1+O(s^2)),\qquad
 \rho b_\lambda'(\rho)=-2aA\rho^{-2a}(1+O(s^2)),
$$
while $B/b_\lambda(\rho)=O(\vartheta_B)$ and
$\rho^2b_\lambda(\rho)^q=O(s^2)$. Integrating in $B$ from zero yields
\begin{equation}\label{in:energy-size}
 -E\asymp AB,\qquad |E_b|\leq C|E|.
\end{equation}
For small $AB>0$ this is a nonconstant positive periodic Fowler
solution, also called a Delaunay solution, see
\cite[Section~2]{MP1999} and \cite[Proposition~2.1]{Santos2010}.
If $\varepsilon$ denotes its minimum, then
$$
 E=-\frac{a^2}{2}\varepsilon^2(1-\varepsilon^q),\qquad
 \varepsilon^2\asymp AB,\qquad |\varepsilon_b|\leq C\varepsilon.
$$
Its metric is complete at the puncture: in cylindrical coordinates
$U_B^{2/a}|\dd x|^2=v_B^{2/a}(\dd t^2+g_{\Sph^{n-1}})$, so escape
along the end has length at least
$\varepsilon^{1/a}\int|\dd t|=\infty$. Small $AB$ also keeps the
profile away from the constant cylindrical solution.

\begin{lemma}\label{in:radial-derivative}\label{in:ode}
For every fixed integer $m\geq0$,
\begin{equation}\label{in:radial-loss}
 |\partial_t^m\partial_bv_B(t)|
       \leq C_m(1+t)v_B(t),\qquad t\geq0.
\end{equation}
For $Q_B=\kappa pU_B^q$, the corresponding scaled H\"older estimates are
\begin{align}
 r^2|Q_B|_{0,\theta;r}
   &\leq C\min\{1,s^2e^{2t}\},\label{in:potential}\\
 r^2|\partial_bQ_B|_{0,\theta;r}
   &\leq C(1+t)\min\{1,s^2e^{2t}\}.
       \label{in:potential-derivative}
\end{align}
The constants are uniform as the neck shrinks and for the phases
selected by \eqref{in:cauchy}.
\end{lemma}

\begin{proof}
Let $w_\varepsilon$ be the periodic profile normalized by
$w_\varepsilon(0)=\varepsilon$ and $w_\varepsilon'(0)=0$; write
$T=T(\varepsilon)$ and $M=M(\varepsilon)$ for its period and maximum.
On the increasing half-period,
$$
 (w_\varepsilon')^2=D_\varepsilon(w_\varepsilon),\qquad
 D_\varepsilon(z)
  =a^2\bigl(z^2-\varepsilon^2-z^{q+2}+\varepsilon^{q+2}\bigr).
$$
Fix a small $c=c(n)>0$ and split the energy integral into
$[\varepsilon,2\varepsilon]$, $[2\varepsilon,c]$, and $[c,M]$.
In the first interval the substitution $z=\varepsilon Z$ gives
$\varepsilon^{-2}D_\varepsilon(\varepsilon Z)\asymp Z^2-1$.
In the middle interval $D_\varepsilon(z)\asymp z^2$. In the last,
$M\to1$ and $|D_\varepsilon'(M)|$ is bounded below; hence the
square-root singularity at $M$ is uniformly integrable. These estimates
give
\begin{equation}\label{in:radial-integrals}
 \int_0^t w_\varepsilon^{q+2}(\tau)\,\dd\tau
       \leq Cw_\varepsilon^{q+2}(t),\qquad
 \int_0^{T/2}w_\varepsilon^q(\tau)\,\dd\tau\leq C
       \quad(0\leq t\leq T/2).
\end{equation}

Set $j=\varepsilon\partial_\varepsilon w_\varepsilon$ and
$h=j/w_\varepsilon$. The variational equation gives
$$
 (w_\varepsilon^2h')'
   =-\kappa q w_\varepsilon^{q+2}h,\qquad h(0)=1,\quad h'(0)=0.
$$
By \eqref{in:radial-integrals},
$$
 |h'(t)|\leq Cw_\varepsilon^q(t)
               \sup_{0\leq\tau\leq t}|h(\tau)|.
$$
Gronwall's inequality therefore bounds $h$ and $h'$ uniformly on the
increasing half-period. Since $|w_\varepsilon'|\leq aw_\varepsilon$,
$$
 |\varepsilon\partial_\varepsilon w_\varepsilon|
  +|\varepsilon\partial_\varepsilon w_\varepsilon'|
       \leq Cw_\varepsilon.
$$
At the maximum, $|w_\varepsilon''(T/2)|$ is uniformly positive.
Differentiating $w_\varepsilon'(T/2)=0$ gives
$|\varepsilon T_\varepsilon|\leq C$. Differentiation of the reflection
identity $w_\varepsilon(t)=w_\varepsilon(T-t)$ extends the relative
bound to the first full period.

The cutoff data satisfy
$$
\begin{aligned}
 v_0:=v_B(0)&\asymp s^a,& |E|/v_0^2&\leq C\vartheta_B,\\
 |\partial_bv_0|/v_0&\leq C\vartheta_B,&
 v_B'(0)/v_0&=a(1+O(s^2+\vartheta_B)).
\end{aligned}
$$
Thus the cutoff lies in the increasing middle interval. If
$w_\varepsilon(\tau_B)=v_0$ with $0<\tau_B<T/2$, then
$w_\varepsilon'(\tau_B)\asymp av_0$ and
$$
 (\tau_B)_b
  =\frac{\partial_bv_0-\varepsilon_b
                   \partial_\varepsilon w_\varepsilon(\tau_B)}
          {w_\varepsilon'(\tau_B)},\qquad |(\tau_B)_b|\leq C.
$$
Also $|T_b|\leq C$, since $|\varepsilon_b|\leq C\varepsilon$.
Now $v_B(t)=w_\varepsilon(t+\tau_B)$. In the $k$th period,
differentiating periodicity introduces the term $-kT_bw_\varepsilon'$;
this is the familiar linear growth of the neck-variation Jacobi field
\cite[Section~5, Lemma~3]{MP1999}; \cite[Section~II]{MPU1996}.
The energy integral on $[c/2,c]$ gives
$T\geq2a^{-1}\log2$ for small $\varepsilon$, so $k\leq C(1+t)$, proving
\eqref{in:radial-loss} for $m=0,1$. The equation and its variational
equation give the higher derivatives; throughout,
$|v_B^{(m)}|\leq C_mv_B$ for each fixed $m$.

Finally, negative energy gives $|v_B'|\leq av_B$ and $v_B\leq1$.
Hence $v_B(t)\leq Cs^ae^{at}$ and
$r^2Q_B=\kappa p v_B^q\leq C\min\{1,s^2e^{2t}\}$.
The parameter estimate follows from \eqref{in:radial-loss}.
One spatial derivative introduces $r^{-1}$, so these pointwise
estimates also give the scaled H\"older bounds.
\end{proof}

\subsection{The local linear inverse}

For each fixed $1<\mu<2$, Silva Santos's Proposition~2.3 and
Corollaries~2.3--2.4 \cite[Section~2.5]{Santos2010} provide a uniformly bounded right inverse
\begin{equation}\label{in:inverse}
 G_{\varepsilon,R,\rho,\boldsymbol z}:Y_\mu\longrightarrow X_\mu^{\mathrm{hi},0},
 \qquad L_{\varepsilon,R,\boldsymbol z}Gf=f,
 \qquad \Pi_{\mathrm{hi}}(Gf)|_{\partial B_\rho}=0,
\end{equation}
where $L=\Delta+\kappa pU_{\varepsilon,R,\boldsymbol z}^q$
. The bound is independent of the small
neck $\varepsilon$, the phase scale $R>0$, the radius $0<\rho\leq1$,
and $\boldsymbol z$ with $|\boldsymbol z|\rho$ sufficiently small.
Here $U_{\varepsilon,R,0}(r)=r^{-a}w_\varepsilon(-\log r+\log R)$,
and $U_{\varepsilon,R,\boldsymbol z}
=\mathcal T_{\boldsymbol z}U_{\varepsilon,R,0}$, with
$\mathcal T_{\boldsymbol z}$ defined in \eqref{in:transform}.
For the Cauchy family, $R=\rho e^{\tau_B}$.
Silva Santos attributes the radial construction to \cite{MP1999} and the
uniformity in $R$ to \cite{Byde2003}. We use the smallest of the three
neck thresholds associated with \eqref{in:weights}.

The right inverse is unique with values in $X_\mu^{\mathrm{hi},0}$. The radial
low-mode Jacobi fields have periodic or linearly growing cylindrical
factors in degree zero and factors $e^{\pm t}$ times periodic functions
in degree one \cite[Section~II]{MPU1996};
\cite[Section~2, equations~(11)--(12)]{KMPS1999}.
Their additive orders are at most $1-a<1$, so none belongs to $X_\mu$.
For a high-mode nullfield, $h=r^aw$ has zero outer trace and decays
at infinity; its cylindrical quadratic form is coercive because
$a^2+2n-\kappa pv^q\geq a+2>0$, where $2n$ is the first high spherical
eigenvalue. Integration by parts therefore gives $h=0$.
For nonzero $\boldsymbol z$, Santos's Corollary~2.4 gives
$$
 \|L_{\varepsilon,R,\boldsymbol z}-L_{\varepsilon,R,0}\|_{X_\mu\to Y_\mu}
       \leq C|\boldsymbol z|\rho.
$$
The boundary projection is unchanged, so injectivity persists by a
Neumann series for small $|\boldsymbol z|\rho$.
Thus $L_{\varepsilon,R,\boldsymbol z}:X_\mu^{\mathrm{hi},0}\to Y_\mu$
is an isomorphism with inverse \eqref{in:inverse}, and uniqueness in the
weaker space identifies its realizations at different weights on their
common domain.

\subsection{Conformal deformations and prescribed high-mode data}

For $\boldsymbol z\in\R^n$, define the special conformal transformation
$$
 D_{\boldsymbol z}(x)
   =1-2\boldsymbol z\cdot x+|\boldsymbol z|^2|x|^2,
 \qquad
 F_{\boldsymbol z}(x)
   =\frac{x-\boldsymbol z|x|^2}{D_{\boldsymbol z}(x)}.
$$
It fixes the puncture and acts on a conformal factor by
\begin{equation}\label{in:transform}
 \mathcal T_{\boldsymbol z}U(x)
   =D_{\boldsymbol z}(x)^{-a}U(F_{\boldsymbol z}(x)),
 \qquad |F_{\boldsymbol z}(x)|=|x|/\sqrt{D_{\boldsymbol z}(x)}.
\end{equation}
This is the deformation of Fowler solutions used in
\cite[Section~2]{KMPS1999} and \cite[Section~2.3]{Santos2010}; conformal
invariance preserves \eqref{eq:flat}. For the dimensionless
parameter $\boldsymbol c=\rho\boldsymbol z$, the bound
$|rU_B'/U_B|\leq2a$ gives
\begin{equation}\label{in:sct-derivative}
 \frac{|\partial_{\boldsymbol c}\mathcal T_{\boldsymbol z}U_B|}
      {\mathcal T_{\boldsymbol z}U_B}\leq Cr/\rho.
\end{equation}
The same estimate holds in scaled spatial norms. For
$|\boldsymbol c|$ in a fixed small neighborhood, the potential estimates
\eqref{in:potential}--\eqref{in:potential-derivative} persist after
transformation, while a $\boldsymbol c$ derivative gains $r/\rho$.

The image radius at $r=\rho$ may exceed $\rho$, so we also record a
two-sided collar estimate. Set $V_B=U_B-b_\lambda$. On a fixed annulus
about $\rho$,
\begin{equation}\label{in:collar}
 |V_B|\leq CB,\qquad
 |r^mV_B^{(m)}|\leq C_ms^2B\quad(m\geq1),
\end{equation}
with identical bounds for $\partial_bV_B$. Indeed,
$$
 V_B''+\frac{n-1}{r}V_B'
  =-\kappa\bigl((b_\lambda+V_B)^p-b_\lambda^p\bigr),
 \qquad V_B(\rho)=B,\quad V_B'(\rho)=0.
$$
After rescaling by $\rho$, the coefficient of the linearized difference
is $O(s^2)$. Since $B/b_\lambda(\rho)$ is small, Gronwall's inequality
on the fixed collar and differentiation of the equation prove
\eqref{in:collar}, including its parameter version.

Let $H_\phi$ be the regular harmonic extension of
$\phi\in\Pi_{\mathrm{hi}}C^{2,\theta}(\Sph^{n-1})$ to $B_\rho$.
Its first possible degree is two, and
\begin{equation}\label{in:harmonic}
 |H_\phi|_{2,\theta;r}
       \leq C\|\phi\|_{C^{2,\theta}}(r/\rho)^2
\end{equation}
\cite[Section~2.6.1, Proposition~2.4]{Santos2010}.
Write $\ell_s=\log(1/s)$. The potential bounds imply, for $1<\mu<2$,
\begin{align}
 \|Q_BH_\phi\|_{Y_\mu}
   &\leq Cs^{2-\mu}\|\phi\|_{C^{2,\theta}},
       \label{in:h-forcing}\\
 \|(\partial_bQ_B)H_\phi\|_{Y_\mu}
   &\leq C(1+\ell_s)s^{2-\mu}\|\phi\|_{C^{2,\theta}}.
       \label{in:h-forcing-derivative}
\end{align}
For example, the second bound is the supremum of
$$
 (1+t)e^{-(2-\mu)t}\min\{1,s^2e^{2t}\},
$$
split at $t=\ell_s$.
These estimates hold for the transformed potential as well; its
$\boldsymbol c$ derivative satisfies \eqref{in:h-forcing}.

\begin{proposition}\label{in:family}
Set $\vartheta_*=B_*\rho^{2a}/A$ and $\phi=B_*\psi$, with
$\psi\in\Pi_{\mathrm{hi}}C^{2,\theta}(\Sph^{n-1})$.
Suppose $|\boldsymbol z|\rho$ lies in the fixed neighborhood of
\eqref{in:inverse}. If $s$, $\vartheta_*$, $AB_*$, and
$\sqrt{\vartheta_*}\|\psi\|_{C^{2,\theta}}$ are sufficiently small,
there is a unique small $W\in X_{\mu_3}$, with zero high-mode boundary
trace, such that
\begin{equation}\label{in:family-equation}
 I=\mathcal T_{\boldsymbol z}U_B+H_\phi+W
\end{equation}
is a positive solution of \eqref{eq:flat}. The correction is
continuously differentiable in $(b,\boldsymbol c,\psi)$ with values in
$X_{\mu_1}$, and its boundary Cauchy data are continuously differentiable
in the usual boundary spaces. Uniformly on a fixed small parameter
neighborhood,
\begin{align}
 \|W\|_{X_{\mu_3}}
   &\leq Cs^{2-\mu_3}B_*\|\psi\|_{C^{2,\theta}},
       \label{in:w-size}\\
 \|\partial_bW\|_{X_{\mu_2}}
   &\leq C(1+\ell_s)s^{2-\mu_2}B_*\|\psi\|_{C^{2,\theta}},
       \label{in:w-b}\\
 \|\partial_{\boldsymbol c}W\|_{X_{\mu_3}}
   &\leq Cs^{2-\mu_3}B_*\|\psi\|_{C^{2,\theta}},
       \label{in:w-c}\\
 \|D_\psi W\|_{\Pi_{\mathrm{hi}}C^{2,\theta}\to X_{\mu_3}}
   &\leq Cs^{2-\mu_3}B_*.
       \label{in:w-psi}
\end{align}
More precisely, if $\mathcal D_{\mathrm{in}}$ is the harmonic
Dirichlet-to-Neumann operator on the unit ball, then
$$
\begin{aligned}
 B_*^{-1}\Pi_{\mathrm{hi}}
       \bigl(I-\mathcal T_{\boldsymbol z}U_B\bigr)|_{r=\rho}&=\psi,\\
 B_*^{-1}\Pi_{\mathrm{hi}}\rho\partial_r
       \bigl(I-\mathcal T_{\boldsymbol z}U_B\bigr)|_{r=\rho}
  &=\mathcal D_{\mathrm{in}}\psi+\mathcal E(b,\boldsymbol c,\psi),
\end{aligned}
$$
where $\mathcal D_{\mathrm{in}}$ has eigenvalue $\ell$ on degree
$\ell\geq2$, and $\mathcal E$ is a small $C^1$ error with the bounds
induced by \eqref{in:w-size}--\eqref{in:w-psi} after division by $B_*$.
\end{proposition}

\begin{proof}
Write $U=\mathcal T_{\boldsymbol z}U_B$, $Q_U=\kappa pU^q$,
$L_U=\Delta+Q_U$, and let $G_U$ be its inverse from
\eqref{in:inverse}. Set
$$
 N_U(Z)=\kappa\bigl((U+Z)^p-U^p-pU^qZ\bigr).
$$
The equation for the correction is
\begin{equation}\label{in:fixed-point}
 W=-G_U\bigl(Q_UH_\phi+N_U(H_\phi+W)\bigr).
\end{equation}
On the ball $\|W\|_{X_{\mu_3}}\leq C_0s^{2-\mu_3}\|\phi\|$,
the lower bound $U\geq c\varepsilon r^{-a}$ and uniform scaled
derivative bounds for $U$ give
$$
 |(H_\phi+W)/U|_{2,\theta;r}
  \leq C\|\phi\|\rho^a/\varepsilon
  \leq C\sqrt{\vartheta_*}\|\psi\|=: \omega.
$$
For $\omega<1/2$, Taylor's formula on $[1-\omega,1+\omega]$ gives,
for every $p>1$,
$$
 |N_U(Z)|\leq C\omega Q_U|Z|,\qquad
 |D_ZN_U(Z)|\leq C\omega Q_U,
$$
with the corresponding scaled H\"older estimates. In particular,
$$
 \|N_U(H_\phi+W)\|_{Y_\mu}
  \leq C\omega\bigl(\|Q_UH_\phi\|_{Y_\mu}
                        +\|W\|_{X_\mu}\bigr).
$$
Taking $\omega$ small relative to the three inverse bounds makes
\eqref{in:fixed-point} a contraction by \eqref{in:h-forcing},
proving \eqref{in:w-size} and positivity. The same contraction works
at each weaker weight in the common ball of radius
$C_0s^{2-\mu_3}\|\phi\|$, since $s^{2-\mu}\leq s^{2-\mu_3}$;
uniqueness identifies these solutions.

For fixed $Z$ in the relative neighborhood above,
\begin{equation}\label{in:nonlinear-background}
 |\partial_bN_U(Z)|
  \leq C(1+t)\omega Q_U|Z|,
\end{equation}
with its scaled H\"older version. This follows by differentiating
Taylor's formula and using \eqref{in:radial-loss}. A
$\boldsymbol c$ derivative has $r/\rho$ in place of $1+t$.
The positive neighborhood can be chosen uniformly for nearby
parameters, because $\varepsilon^2\asymp AB$ and $B/B_*$ stays in a
fixed compact interval.

We first treat $b$. Keep $\boldsymbol c,\psi$ fixed and set
$$
 U^{(b)}=\mathcal T_{\boldsymbol z}U_{B_*e^b},\qquad
 Q^{(b)}=\kappa p(U^{(b)})^q,\qquad
 L^{(b)}=\Delta+Q^{(b)}.
$$
Write $W_b=W(b,\boldsymbol c,\psi)$, $W_h=W(b+h,\boldsymbol c,\psi)$,
$Z_h=H_\phi+W_h$, and $Z=H_\phi+W_b$. Set
$$
 C_h=\int_0^1D_ZN_{U^{(b)}}\bigl(Z+\tau(Z_h-Z)\bigr)\,\dd\tau.
$$
Its multiplier norm $X_\mu\to Y_\mu$ is at most $C\omega$.
Subtracting the equations gives
\begin{equation}\label{in:parameter-difference}
\begin{aligned}
 (L^{(b)}+C_h)(W_h-W_b)
  &=-(Q^{(b+h)}-Q^{(b)})(H_\phi+W_h)\\
  &\quad-\bigl(N_{U^{(b+h)}}(Z_h)-N_{U^{(b)}}(Z_h)\bigr).
\end{aligned}
\end{equation}
All three realizations of $L^{(b)}+C_h$ are uniformly invertible by a
Neumann series. To estimate the source, put $d=\mu_3-\mu_2$. The
potential envelope gives
$$
 \sup_{t\geq0}(1+t)e^{-dt}\min\{1,s^2e^{2t}\}
       \leq C_d(1+\ell_s)s^d.
$$
Indeed, write $t=\ell_s-u$ before the transition and $t=\ell_s+u$ after it.
The two factors to bound are respectively
$(1+\ell_s)s^de^{-(2-d)u}$ and
$s^d(1+\ell_s+u)e^{-du}$.
Thus
$$
 \|(\partial_bQ_U)W_h\|_{Y_{\mu_2}}
   \leq C(1+\ell_s)s^{\mu_3-\mu_2}\|W_h\|_{X_{\mu_3}}
   \leq C(1+\ell_s)s^{2-\mu_2}\|\phi\|.
$$
For the harmonic term use \eqref{in:h-forcing-derivative}; for the
nonlinear difference use \eqref{in:nonlinear-background} and the
same multiplier estimate. Integration along the parameter interval
in \eqref{in:parameter-difference} yields local Lipschitz continuity
into $X_{\mu_2}$ and
$$
 \left\|(W_h-W_b)/h\right\|_{X_{\mu_2}}
     \leq C(1+\ell_s)s^{2-\mu_2}B_*\|\psi\|_{C^{2,\theta}}.
$$

To obtain differentiability, divide
\eqref{in:parameter-difference} by $h$. On each compact annulus the
source converges in $C^{0,\theta}$ to
$$
 F_b=-(\partial_bQ^{(b)})(H_\phi+W_b)
       -\partial_bN_{U^{(b)}}(H_\phi+W_b),
$$
where the last derivative holds its argument fixed. The tails are
uniformly small in $Y_{\mu_1}$: the correction terms are bounded by
$(1+t)e^{-(\mu_3-\mu_1)t}$ and the harmonic terms by their degree-two
decay. Also $C_h\to D_ZN_{U^{(b)}}(Z)$ on compact annuli; acting on the
bounded $X_{\mu_2}$ difference quotients, its tails gain
$e^{-(\mu_2-\mu_1)t}$. Applying the fixed inverse at $\mu_1$ proves
convergence of the difference quotients to the unique solution of
\begin{equation}\label{in:parameter-derivative}
 \bigl(L_U+D_ZN_U(Z)\bigr)\partial_bW
  =-(\partial_bQ_U)(H_\phi+W)-\partial_bN_U(Z),
 \qquad Z=H_\phi+W.
\end{equation}
In particular the coefficient derivative acts on $W$ as well as on
$H_\phi$. The source belongs to $Y_{\mu_2}$; compatibility of inverses
gives \eqref{in:w-b}. Comparing this equation at nearby parameters,
using convergence on compact annuli and the same tail bounds, proves
continuity of $\partial_bW$ into $X_{\mu_1}$.

For $\boldsymbol c$, the additional factor $r/\rho$ makes coefficient
differentiation bounded on a single weighted space. The same
subtraction argument at $\mu_3$ gives \eqref{in:w-c}. Differentiation
in $\psi$ gives
$$
 \bigl(L_U+D_ZN_U(Z)\bigr)D_\psi W[\dot\psi]
   =-\bigl(Q_U+D_ZN_U(Z)\bigr)H_{B_*\dot\psi},
$$
which proves \eqref{in:w-psi}. These comparison equations give joint
continuity of the derivatives into $X_{\mu_1}$.

Restriction to the outer annulus is bounded from $X_{\mu_1}$ to
$C^{2,\theta}$, and $\rho\partial_r$ maps its trace to
$C^{1,\theta}$. Since $\Pi_{\mathrm{hi}}W|_{r=\rho}=0$, the claimed
boundary error is exactly
$\mathcal E=B_*^{-1}\Pi_{\mathrm{hi}}\rho\partial_rW|_{r=\rho}$.
The estimates and their parameter versions follow from the corresponding
bounds for $W$.
\end{proof}

\subsection{The corresponding exterior bubble}

The exterior bubble must undergo the same conformal transformation as
the local model. The exact formula is
\begin{equation}\label{in:bubble-transform}
\begin{gathered}
 \mathcal T_{\boldsymbol z}b_\lambda(x)
  =\left(\frac{2\lambda'}{|x-\xi|^2+(\lambda')^2}\right)^a,\\
 C=1+\lambda^2|\boldsymbol z|^2,\quad
 \lambda'=\lambda/C,\quad \xi=\lambda^2\boldsymbol z/C.
\end{gathered}
\end{equation}
To see this, complete the square in
$\lambda^2D_{\boldsymbol z}(x)+|x|^2$. Its coefficient at infinity
is $AC^{-a}$, rather than $A$. Nevertheless, for $r\geq\rho$ and
$\boldsymbol c=\rho\boldsymbol z$ in the fixed neighborhood,
$$
 |\partial_{\boldsymbol c}\mathcal T_{\boldsymbol z}b_\lambda|
        \leq Cs^2\mathcal T_{\boldsymbol z}b_\lambda.
$$

Subtracting the identical transformed bubbles across the interface
leaves the exact difference
$$
 \mathcal T_{\boldsymbol z}U_B
       -\mathcal T_{\boldsymbol z}b_\lambda
       =\mathcal T_{\boldsymbol z}V_B.
$$
By \eqref{in:collar}, its Cauchy data and their parameter derivatives
are bounded after division by $B$. At $\boldsymbol z=0$, its
Dirichlet trace is $B$ and its radial derivative vanishes. For
$\boldsymbol\omega=x/|x|$,
$$
 \left.\partial_{\boldsymbol c}
        (\mathcal T_{\boldsymbol z}V_B)\right|_{r=\rho,\,\boldsymbol z=0}
       [\dot{\boldsymbol c}]
      =2aB\,\dot{\boldsymbol c}\cdot\boldsymbol\omega,
$$
and
$$
\begin{aligned}
 \left.\rho\partial_r\partial_{\boldsymbol c}
        (\mathcal T_{\boldsymbol z}V_B)\right|_{r=\rho,\,\boldsymbol z=0}
       [\dot{\boldsymbol c}]
  &=\bigl(2aB+\rho^2V_B''(\rho)\bigr)
                 \dot{\boldsymbol c}\cdot\boldsymbol\omega\\
  &=2aB\,\dot{\boldsymbol c}\cdot\boldsymbol\omega
       +O(s^2B|\dot{\boldsymbol c}|).
\end{aligned}
$$
The unnormalized radial derivative has the additional factor
$\rho^{-1}$. If $|\boldsymbol z|\rho=\sigma$ is small, the normalized
first derivative formulas differ from those at zero by
$O(B(\sigma+s^2))$. These estimates, together with those of
Proposition~\ref{in:family}, are used in the proof of
Lemma~\ref{match:estimate}.

\section{Boundary operators for the old ends}
\label{old:section}

Fix a positive solution $U_0$ with a finite singular set $F$, noncylindrical Delaunay ends, and trivial $L^2(g_{U_0})$ kernel of
the relative Jacobi operator $\Delta_{g_{U_0}}+n$, where
$g_{U_0}=U_0^{2/a}|\dd x|^2$. This is marked
nondegeneracy in \cite[Definition 3]{MP1999}. Constants $C_{\rm old}$ may
depend on $U_0$, the allowed center set, and the boundary sections;
the Newtonian singular coefficient depends only on the dimension.

Choose a stereographic pole outside $F$ and the allowed center set $K$.
In the resulting flat coordinates put
$$
 L_0=\Delta+V_0,\qquad V_0=\kappa p U_0^{p-1}.
$$
Cut each old end sufficiently far out and remove a small regular cap
about the pole, leaving a bounded smooth domain $D\subset\R^n$ with
$K\Subset D$. Its boundary $\Sigma$ consists of the old-end and cap sections. The normal $\mathbf n$ points out of $D$. Neumann traces from the
eliminated pieces are taken in this same direction, opposite to their
domain outward normal. The symbol $\nu$ remains reserved for the
ratio $\rho_i/r_i$.

For the prescribed core $D_{\rm pre}$ of
Theorem~\ref{thm:insertion}, these sections are moved towards
the punctures and the cap is shrunk so that $D\supset D_{\rm pre}$
and the auxiliary invertibility conditions below hold.
The symbol $D$ denotes this fixed working core throughout the
finite gluing proof; $K$ is unchanged. The radius and coefficient
thresholds may depend on the working core. Only the resulting
global solution and its lower bound are restricted to $D_{\rm pre}$.

\subsection{A common complement for the deficiency spaces}

On an old end write $r=e^{-t}$ and $U_0=r^{-a}(v_\epsilon(t+\phi)+R)$.
The remainder and its cylindrical derivatives decay exponentially
\cite[Theorem~3.3]{MPU1996}; the finer conformal expansion is given in
\cite[Theorem~1]{KMPS1999}.
Fix two numbers
\begin{equation}\label{old:weights}
 0<d_-<d_+<\min\{a,1,\gamma_{\rm old}\},
\end{equation}
where $\gamma_{\rm old}>0$ is a common exponential decay rate for the old remainders.
Write $\mathcal C_d^{m,\theta}$ for the cylindrical H\"older space with
decay $e^{-dt}$. For an additive flat Jacobi field $J$, impose this decay
on $r^aJ$, equivalently on $J/U_0$ since the fixed positive Delaunay
profile is bounded above and below. The cylindrical source is
$r^{a+2}L_0J$. At the regular stereographic pole we require smoothness
in the round gauge. In particular, for $J=U_0j$, conformal covariance gives
$$
 L_0(U_0j)=U_0^p(\Delta_{g_{U_0}}+n)j.
$$
The flat and relative realizations therefore have equivalent weighted
norms on each old end.

Let $m=|F|$. The \emph{deficiency space} $\mathcal W$ is the
$2m$-dimensional space of radial coefficients recording the infinitesimal
neck and phase changes on the old ends. We identify it with a space of
cutoff radial Jacobi fields; their Jacobi residuals decay exponentially,
possibly with polynomial factors. We use the standard deficiency space of
\cite[Section~IV]{MPU1996}. The finite-end
Fredholm theory of \cite[Sections~IV--V]{MPU1996} gives
\begin{equation}\label{old:fredholm}
 L_0:\mathcal C_d^{2,\theta}\oplus\mathcal W
       \longrightarrow\mathcal C_d^{0,\theta}
 \quad\hbox{surjective},
 \qquad \dim\ker L_0=m,
\end{equation}
with the understood cylindrical conjugation on the ends. The operator
on the pure decaying space is injective and has index $-m$. The radial
coefficient map on the extended kernel is injective. Denote its image
in $\mathcal W$ by $Z$.

For the weighted H\"older realization, the same asymptotic decomposition
and stripwise Schauder estimates give the corresponding bounds
\cite[Chapters~6 and~9]{GT2001}.
The absence of indicial exponents in $(0,1)$ restores the prescribed
decay weight.
A fixed boundary-collar extension handles Dirichlet data, and all
inverse constants depend on the fixed old solution, sections, and
chosen weights.

\begin{lemma}\label{old:common-complement}
The old cross sections can be chosen so that the disjoint union $E$
of the old tails has the following properties. Its Dirichlet
operator on the pure decaying space is injective and has index $-m$.
After adjoining $\mathcal W$ it is surjective, with an $m$-dimensional
kernel whose radial coefficient image $N_{\rm tail}$ is a $m$-plane in
$\mathcal W$. There is a single $m$-plane $\mathcal W'$ complementary
to both $Z$ and $N_{\rm tail}$. The whole old operator and the tail Dirichlet operator restricted to decay plus $\mathcal W'$ are
isomorphisms for every $0<d\leq d_+$.
\end{lemma}

\begin{proof}
First consider an exact radial Delaunay half-cylinder. Its conjugated
Jacobi equation is
$$
 w_{tt}+\Delta_{\Sph^{n-1}}w-a^2w+\kappa p v_\epsilon^{p-1}w=0.
$$
The radial equation has two independent Jacobi fields with Floquet
multiplier one; one may grow linearly in $t$. Neither decays. The decaying first spherical field
$e^{-t}(av_\epsilon-v_\epsilon')$ has a strictly positive periodic
factor, hence nonzero boundary trace. For degrees at least two,
$$
 a^2+2n-\kappa p v_\epsilon^{p-1}\geq a+2>0.
$$
Integration by parts on the half-cylinder therefore excludes a
decaying field of spherical degree at least two with zero Dirichlet trace. Boundary terms at
infinity vanish by the decay and local elliptic estimates.

Radial variation of constants, the explicit first spherical modes,
and the coercive estimate for degrees at least two give the
inhomogeneous bounds. Local elliptic estimates then give the required
H\"older regularity. In the radial
sector decay imposes one compatibility condition on the source with
zero trace; adjoining the two neutral fields removes that condition
and leaves one kernel direction. Every nonradial decaying Dirichlet problem is invertible. Thus each
end has index $-1$ before extension, and contributes one kernel
direction after the surjective radial extension.

For the original tails the difference from the model operator is
exponentially small after the cuts are moved out. The preceding
estimates and a Neumann series preserve pure-decay injectivity and
extended surjectivity. Index stability gives kernel dimension $m$.
A kernel field with zero radial coefficients is pure decaying, so its
coefficient map is injective. Its image is $N_{\rm tail}$.

Choose $\mathcal W'$ in the intersection of the two open dense
sets of $m$-planes complementary to $Z$ and $N_{\rm tail}$.
Removing the respective kernel by this common complement gives both
isomorphisms. Changing $d$ within $(0,d_+]$ crosses no
indicial exponent and stays below the coefficient decay rate, so the
coefficient planes and the choice of complement are unchanged.
The Sobolev to H\"older argument above gives the asserted realizations.
\end{proof}

The plane $\mathcal W'$ may couple ends, so its boundary operator need
not be selfadjoint. Below, only the compact full Dirichlet realization
uses selfadjointness.

\subsection{The nonlinear boundary map}

The regular cap has no radial coefficient space. Its bounded Jacobi
potential and the divergent first Dirichlet Laplace eigenvalue as the
cap shrinks give Dirichlet invertibility for a small cap.
The Dirichlet implicit function theorem applies, with trace and flux
converted to flat core coordinates by conformal covariance.

\begin{proposition}\label{old:nonlinear-graph}\label{old:graph}
For full boundary increments $\gamma\in C^{2,\theta}(\Sigma)$ near
zero, the eliminated tails and cap have a unique positive solution in a
fixed neighborhood of the old solution in the period-aligned coordinates,
with old radial parameter increment in $\mathcal W'$. Its physical
Neumann trace has the form
\begin{equation}\label{old:graph-expansion}
\begin{aligned}
 \mathcal N_{\mathrm{old}}(\gamma)&=\mathcal N_{\mathrm{old}}(0)+T\gamma+\mathcal E(\gamma),\\
 \|\mathcal E(\gamma)\|_{C^{1,\theta}}
       &\leq C_{\rm old}\|\gamma\|_{C^{2,\theta}}^2,\\
 \|D\mathcal E(\gamma)\|
       &\leq C_{\rm old}\|\gamma\|_{C^{2,\theta}}.
\end{aligned}
\end{equation}
This is a $C^2$ boundary map. At a sufficiently small boundary
increment its derivative represents every Jacobi field on the
eliminated pieces whose radial parameter derivative belongs to
$\mathcal W'$ and whose physical remainder has decay $d_-$. In
particular it represents every Jacobi field with zero radial coefficients and positive
exponential decay there.
\end{proposition}

\begin{proof}
Let $w=(\epsilon_j-\epsilon_j^0,\phi_j-\phi_j^0)_j\in\mathcal W'$.
Choose an absolute cylindrical factor $A_w$ equal to the old solution
on a boundary collar and to the new radial model plus the fixed old
decaying remainder far out.
Use $t=F_w(s)$, equal to the identity on that collar and to
$c_j(w)s+b_j(w)$ far out on the $j$th end, where
$$
 c_j(w)=T(\epsilon_j)/T(\epsilon_j^0).
$$
Choose $b_j$ smoothly and place the transition on a fixed compact
interval separated from the collar.
The pulled operator on $V(s,\omega)$ is
\begin{equation}\label{old:pulled-operator}
 F_s^{-2}V_{ss}-F_{ss}F_s^{-3}V_s+
 \Delta_{\Sph^{n-1}}V-a^2V+\kappa V^p.
\end{equation}
With $z_\epsilon(\vartheta)=v_\epsilon(T(\epsilon)\vartheta/(2\pi))$,
the far-out leading profile is
$$
 z_{\epsilon_j}\left(\frac{2\pi s}{T(\epsilon_j^0)}+
              \frac{2\pi(b_j+\phi_j)}{T(\epsilon_j)}\right).
$$
It solves the pulled equation there and has bounded parameter
derivatives on the fixed period circle near each positive old neck.

Let $H\gamma$ be a bounded cylindrical extension supported in the
collar where $F_w$ is the identity, with flat-section trace
$r^a\gamma$ and the fixed conformal conversion included.
Write $V=A_w\circ F_w+H\gamma+e$, with
$e\in\mathcal C_{d_+,0}^{2,\theta}$; the subscript $0$ denotes zero
full boundary trace.
The pulled residual $\mathcal F(e,w,\gamma)$ is $C^2$ into the
weighted source space.
Indeed, the variable remainder $e$ is never composed with $F_w$,
so parameter differentiation requires no derivatives beyond $e_{ss}$.
Only the fixed smooth old remainder is composed with $F_w$; its first
two parameter derivatives acquire polynomial factors of degree at most
two, absorbed by the strict decay gap in \eqref{old:weights} after
shrinking the parameter neighborhood. The fixed positive profiles
preserve positivity nearby.
At the base point, the derivative in $(e,w)$ is the physical Dirichlet
Jacobi operator, up to an invertible triangular change of remainder:
the coordinate derivative of the base residual vanishes.
Lemma~\ref{old:common-complement} therefore makes it an isomorphism.
The implicit function theorem gives a positive $C^2$ solution graph,
and the fixed collar coordinates make its Neumann trace a bounded
linear observation. Taylor's theorem gives \eqref{old:graph-expansion}.

For the physical tangent identification, let $c_{\min}$ and
$c_{\max}$ bound the slopes $c_j(w)$ throughout a smaller neighborhood,
and arrange
$$
 c_{\max}d_-<d_+,\qquad
 c_{\min}\min\{1,\gamma_{\rm old}\}>d_+,\qquad
 0<\widetilde d<\min\{d_+,c_{\min}d_-\}.
$$
Fix $\widetilde d$ on this neighborhood.
Lemma~\ref{old:common-complement} gives the base isomorphism at
this weight with the same complement.
At a current graph point $(e_*,w_*,\gamma_*)$, the derivative extends to
$$
 \mathcal A_*:
 \mathcal C_{\widetilde d,0}^{2,\theta}\oplus\mathcal W'
 \longrightarrow\mathcal C_{\widetilde d}^{0,\theta}.
$$
Period alignment keeps the leading profiles and principal coefficients
uniformly close to their base values. The parameter derivatives of the
fixed old remainder decay faster than $d_+$, and the terms involving
$e_*$ decay at $d_+$, so the finite-dimensional derivatives are close
in the $\widetilde d$ source norm as well.
Thus $\mathcal A_*$ is a small perturbation of the base isomorphism
and remains invertible.
Graph derivatives have physical remainders of decay $d_-$: differentiation
at fixed physical $t$ adds at most a factor $1+t$ to the solved
remainder, while $d_+/c_{\max}>d_-$.

Let $J$ be a physical Jacobi field with the stated decay and radial
coefficients in $\mathcal W'$. Subtract the graph derivative with its
full Dirichlet trace, obtaining a zero-trace field $J^0$ with
$h=r^aJ^0$ and radial coefficient vector $\zeta\in\mathcal W'$.
On the far tails,
$$
 h(t)=\partial_w v_{w_*}(t)\zeta+h_{\rm dec}(t),
 \qquad h_{\rm dec}\in\mathcal C_{d_-}^{2,\theta},
$$
where $v_w$ is the physical leading radial model.
Set $F=F_{w_*}$ and $\dot F_\zeta=D_wF_{w_*}\zeta$.
For the current physical factor $V_{\rm phys}$, differentiating
$V_{\rm phys}\circ F=A_{w_*}\circ F+H\gamma_*+e_*$ gives the
period-aligned remainder derivative
$$
 \widehat e_{J^0}
 =h\circ F-(\partial_w A_{w_*})\circ F\,\zeta
   +\frac{\dot F_\zeta}{F_s}\,\partial_s e_* .
$$
The collar extension has zero derivative for this difference field.
Far out, $\partial_w A_{w_*}=\partial_wv_{w_*}$, so the first two
terms reduce to $h_{\rm dec}\circ F$, with decay at least
$c_{\min}d_-$; the last is $O((1+s)e^{-d_+s})$ in
$C^{2,\theta}$ on unit strips.
For this last assertion, $\dot F_\zeta$ is supported away from the
boundary collar, where the equation has smooth coefficients and no
collar forcing. Interior Schauder estimates and differentiation of the
equation give the additional spatial derivatives of $e_*$ with the
same decay, without extra boundary regularity of $\gamma_*$.
Near the boundary, $F$ is the identity, $\dot F_\zeta=0$, and
$A_w$ is independent of $w$. Hence
$\widehat e_{J^0}\in\mathcal C_{\widetilde d,0}^{2,\theta}$.
Differentiation at zero residual yields
$$
 \mathcal A_*(\widehat e_{J^0},\zeta)=0.
$$
Invertibility forces both entries to vanish, so $J^0=0$.
This proves the stated tangent identification. Marked fields have zero
radial coefficients and decay at every rate below the first nonradial
exponent $1$, so they are included; the cap uses compact Dirichlet
uniqueness.
\end{proof}

\subsection{The inverse on a perforated core}

Let $B(y_i,r_i)$, $i=1,\ldots,N$, be disjoint original balls with centers in $K$, and
put $\rho_i=\nu r_i$, where $0<\nu<1/4$. Denote
$$
\begin{gathered}
 D_N=D\setminus\bigcup_i\overline{B(y_i,\rho_i)},\qquad
 c_N=\sum_i\rho_i^{n-2},\\
 \Gamma_i=\partial B(y_i,\rho_i),\qquad
 \Gamma_N=\bigcup_{i=1}^N\Gamma_i.
\end{gathered}
$$
Up to dimensional normalization, $c_N$ sums the individual capacities
and bounds the union's capacity by subadditivity
\cite[Chapter~2]{AdamsHedberg1996}.

\begin{proposition}\label{old:perforated-inverse}\label{old:inverse}
For sufficiently small $c_N$ the problem
\begin{equation}\label{old:linear-boundary-problem}
 L_0v=f\text{ in }D_N,\qquad v=g\text{ on }\Gamma_N,
 \qquad \partial_{\mathbf n} v-T(v|_\Sigma)=r\text{ on }\Sigma
\end{equation}
has a unique solution. For every fixed $U\Subset D$ containing all
new balls its zero-new-trace source inverse $G_N^0$ has a kernel
satisfying
\begin{equation}\label{old:green-kernel}
 |G_N^0(x,z)|\leq C_n|x-z|^{2-n}+C_{\rm old},
 \qquad x,z\in U\cap D_N.
\end{equation}
If the source is supported in $U$ and separated from a fixed old
collar, the $C^{2,\theta}$ norm of the solution on that collar is
at most $C_{\rm old}\|f\|_{L^1}$. Smooth interface forcing has
a uniformly bounded $C^0$ response on $U\cap D_N$, with uniform
scaled local $C^{2,\theta}$ estimates on the new collars and interior
patches. Ordinary $C^k$ estimates hold on compact subsets staying
a fixed positive distance from the holes. All these constants are
independent of the number of holes and their radius ratios.
\end{proposition}

\begin{proof}
Keep the old sections fixed and choose a sufficiently small regular cap
with invertible Dirichlet Jacobi operator.
An arbitrarily small outward change of its outer radius makes the full Dirichlet
operator on $D$ invertible: strict domain monotonicity, by min--max and
unique continuation \cite{Aronszajn1957}, leaves only countably many
exceptional radii. The cap radius is chosen within the preceding
working-core convention, preserving $D\supset D_{\rm pre}$.

Put $B_{D,N}=-\Delta-V_0$ with full Dirichlet data. Choose
$\chi_i=1$ on $B(y_i,\rho_i)$, supported in $B(y_i,2\rho_i)$, with
$\int|\nabla\chi_i|^2\leq C_n\rho_i^{n-2}$, in a fixed neighborhood
separated from $\Sigma$. Their supports are disjoint
because $2\rho_i+2\rho_j<r_i+r_j$. For $\chi=\sum_i\chi_i$,
multiplication by $1-\chi$ changes each fixed smooth trial function by
$O(c_N^{1/2})$ in $H^1$; the support volume is at most
$C_n\rho_{\max}^2c_N$. Min--max and domain monotonicity therefore give
convergence of the fixed eigenvalues and a uniform gap about zero,
as in the capacitary perturbation method of \cite{RauchTaylor1975}.
The uniformly coercive shifted form, with
$c_{\rm shift}>\|V_0\|_\infty+1$, and spectral calculus give a uniform
$H^{-1}\to H^1$ inverse bound.

For sufficiently small fixed $t_0>0$, use the signed-resolvent identity
$$
 B_{D,N}^{-1}=\int_0^{t_0}e^{-tB_{D,N}}\,\dd t+
 e^{-t_0B_{D,N}/2}B_{D,N}^{-1}e^{-t_0B_{D,N}/2}.
$$
Dirichlet heat-kernel domination \cite[Chapters~1 and~3]{Davies1989}
bounds the first kernel by $C_n|x-z|^{2-n}$, choosing
$e^{t_0\|V_0\|_\infty}\leq e$. The gap and the Gaussian $L^2$ row
bounds make the second kernel bounded by $C_{\rm old}$; no positivity
of the resolvent is required.

Let $D_N^\Sigma$ be the core Dirichlet-to-Neumann map with zero new
trace. The no-hole Poisson solution $u_0$ for
$\varphi\in H^{1/2}(\Sigma)$ has uniformly bounded interior jets near
the holes. Correcting $(1-\chi)u_0$ by the preceding inverse gives,
after zero extension across the holes,
$$
 \|u_N-u_0\|_{H^1(D)}
 \leq C_{\rm old}\|\chi u_0\|_{H^1(D)}
 \leq C_{\rm old}c_N^{1/2}\|\varphi\|_{H^{1/2}(\Sigma)}.
$$
The difference is homogeneous with zero trace on a fixed collar of
$\Sigma$, so boundary regularity gives, for every fixed $s$,
$$
 \|D_N^\Sigma-D_0^\Sigma\|_{H^{1/2}(\Sigma)\to H^s(\Sigma)}
 \leq C_{{\rm old},s}c_N^{1/2}.
$$

Weighted regularity extends $T:H^{1/2}\to H^{-1/2}$. Remote-end
responses are smoothing, and changing $\mathcal W'$ adds finite-rank
smooth terms. The local calculus \cite[Theorems~11.17--11.18]{Grubb2009}
therefore gives $D_0^\Sigma-T$ principal symbol $2|\xi|$ and index
zero by deformation to $2(1-\Delta_\Sigma)^{1/2}$. Its kernel glues
to a whole old Jacobi field with coefficients
in $\mathcal W'$, excluded by Lemma~\ref{old:common-complement}.
Thus $D_N^\Sigma-T$ remains uniformly invertible by the smoothing
perturbation, without selfadjointness of $T$.

Solve the full Dirichlet problem and correct its old normal trace
using $(D_N^\Sigma-T)^{-1}$ and the old-boundary Poisson operator.
A fixed-collar extension of smooth old data and the separated
Dirichlet kernel bound give this Poisson operator a uniform bound on
$U$. For a source pole in $U$, fixed-collar estimates bound its old
normal trace and hence its interface correction by $C_{\rm old}$.
This proves \eqref{old:green-kernel}; differentiation on fixed collars
gives the stated separated-source and interface estimates.
Finally, the new collars are disjoint fixed-shape annuli after
rescaling, with uniformly small smooth potential
$\rho_i^2V_0(y_i+\rho_i\cdot)$. Local boundary and interior estimates
\cite[Chapters~6 and~8]{GT2001} therefore have constants independent
of $N$ and all radius ratios.
\end{proof}

\subsection{The Poisson operator}

Let $P_Ng$ solve \eqref{old:linear-boundary-problem} with $f=r=0$.
Boundary functions on a new sphere are measured after pullback to
the unit sphere. If $b_i=\|g_i\|_{C^{2,\theta}(\Sph^{n-1})}$, put
$q_g=\sum_i b_i\rho_i^{n-2}$.

\begin{lemma}\label{old:poisson}
For all $x\in D_N$,
\begin{equation}\label{old:poisson-bound}
 |P_Ng(x)|\leq C_n\sum_i
   \frac{b_i\rho_i^{n-2}}{(\rho_i^2+|x-y_i|^2)^{(n-2)/2}}
      +C_{\rm old}q_g.
\end{equation}
Its fixed old-collar $C^{2,\theta}$ norm is at most
$C_{\rm old}q_g$. The same pointwise envelope controls scaled local
$C^{2,\theta}$ norms on the new collars and on interior patches
whose size is a fixed small multiple of their distance to the holes,
capped at a fixed old radius.
\end{lemma}

\begin{proof}
First solve with new data $g$ and old data zero using the selfadjoint
compact full Dirichlet inverse. For $x\notin B(y_i,2\rho_i)$, its
Green kernel, as a function of the source variable $z$, has zero trace
and solves the homogeneous equation on
$B(y_i,3\rho_i/2)\setminus\overline{B(y_i,\rho_i)}$.
Rescaled boundary gradient estimates and the Dirichlet kernel bound give
$$
 |\partial_{\mathbf n_z}G_{D,N}(x,z)|
 \leq C\rho_i^{-1}(|x-y_i|^{2-n}+C_{\rm old}),
 \qquad z\in\partial B(y_i,\rho_i).
$$
Integration over the sphere contributes $b_i\rho_i^{n-2}$.
On doubled new collars, the prescribed trace and the outer-boundary
bound give the same envelope by scaled Dirichlet estimates.
The annuli have fixed shape, avoid the other holes, and have uniformly
small scaled old potential; their constants are independent of the
number of holes and the radius ratios.

At $\Sigma$, separated-source estimates bound the normal trace by
$C_{\rm old}q_g$ in the needed fixed norms. Correct it using
$(D_N^{\Sigma}-T)^{-1}$ and the old-boundary Poisson operator.
The correction is $O(C_{\rm old}q_g)$ on the new region and fixed old
collars, and has zero new trace. Boundary and interior Schauder
estimates on the stated patches give the scaled local bounds, since
the envelope varies by only a fixed factor there. This also proves the
fixed old-collar estimate.
\end{proof}

For later use, let $F$ be a fixed $C^{0,\theta}$ source supported a
positive distance from the new-center region. Let $G_0$ denote the
no-hole source inverse with homogeneous old interface condition and
radial coefficients in $\mathcal W'$, and put $w_0=G_0F$. Then the exact identity
\begin{equation}\label{old:regular-source}
 G_N^0F=w_0-P_N(w_0|_{\Gamma_N})
\end{equation}
holds. The old-collar correction is bounded by
$C_{\rm old}c_N\|F\|_{C^{0,\theta}}$.
The traces of $w_0$ have uniformly bounded scaled $C^{2,\theta}$ norms,
so this identity also gives uniform bounds for regular sources on fixed compact subsets.

\subsection{The normalized Dirichlet-to-Neumann map}

Suppose the positive frozen amplitudes $A_i,B_i$ satisfy
\begin{equation}\label{old:rows}
 B_i\rho_i^{2a}\leq C\nu^{2a}A_i,\qquad
 \sum_{j\ne i}\frac{A_j}{|y_i-y_j|^{2a}}\leq CB_i,
 \qquad B_i\geq m_{\rm old}>0.
\end{equation}
The constants $C$ in the first two inequalities are fixed universal
constants. Put $Q=\sum_iA_i$ and use the product norms
$$
 \|g\|_{\mathcal B}=\max_i B_i^{-1}
       \|g_i\|_{C^{2,\theta}(\Sph^{n-1})},\qquad
 \|f\|_{\mathcal B_1}=\max_i B_i^{-1}
       \|f_i\|_{C^{1,\theta}(\Sph^{n-1})}.
$$
Thus $q_g\leq C\nu^{2a}Q\|g\|_{\mathcal B}$.
For the spatial radial coordinate $R_i(x)=|x-y_i|$, the derivative
on a new sphere in the following proposition is $+\partial_{R_i}$,
away from $y_i$. It is the negative of the perforated core's domain
outward normal.

For the nonlinear Poisson estimates, set
\begin{equation}\label{eq:weights}
\begin{gathered}
 h_i(x)=\frac{A_i}{(\rho_i^2+|x-y_i|^2)^a},\qquad
 h=\sum_i h_i,\qquad H_i=h-h_i,\\
 \Psi_i=U_0+H_i,\qquad b=Q^{1/n},\qquad \Phi=h+bU_0.
\end{gathered}
\end{equation}
Suppose, as holds for the frozen packed backgrounds, that
$\Psi_i\geq cB_i$ on its own sphere, with universal $c>0$.
The bubble identity $h_i=A_i(2\rho_i)^{-a}b_{\rho_i,y_i}$
gives smoothness and superharmonicity, as in
Lemma~\ref{pot:positive}.
The harmonic monopole $B_i\rho_i^{2a}|x-y_i|^{-2a}$ is bounded by
$C\Psi_i$ on $\Gamma_i$, and by $C_{\rm old}\nu^{2a}Q$ on $\Sigma$.
After decreasing $Q$, comparison on $D$ minus only the $i$th hole gives
its domination by $C\Psi_i$ everywhere; the other tails are smooth
across their centers. By \eqref{old:rows},
$$
 \sum_j\frac{B_j\rho_j^{2a}}{(\rho_j^2+|x-y_j|^2)^a}
       \leq C\nu^{2a}h(x),\qquad
 \sum_{j\ne i}\frac{B_j\rho_j^{2a}}{(\rho_j^2+|x-y_j|^2)^a}
       \leq C\nu^{2a}H_i(x).
$$
Lemma~\ref{old:poisson} therefore gives
\begin{equation}\label{old:weighted-poisson}
\begin{aligned}
 \sup_{D_N}\frac{|P_Ng|}{\Phi}
   &\leq C\nu^{2a}\|g\|_{\mathcal B},\\
 \max_i\sup_{D_N}\frac{|P_Ng|}{\Psi_i}
   &\leq C\|g\|_{\mathcal B},\\
 \|P_Ng\|_{C^{2,\theta}(\text{old collars})}
   &\leq C_{\rm old}\nu^{2a}Q\|g\|_{\mathcal B}.
\end{aligned}
\end{equation}
The bounded term is absorbed by $bU_0$ for $\Phi$ and by $U_0$ for
$\Psi_i$, since $Q/b=Q^{1-1/n}\to0$. Fix $\nu$ before decreasing $Q$.
Weight comparability, bounded relative jets, and the scaled local
estimates of Lemma~\ref{old:poisson} give the same local bounds.
In the weighted norms \eqref{eq:space-X},
\begin{equation}\label{old:poisson-X}
 \|P_Ng\|_{X_\Phi}\leq C\nu^{2a}\|g\|_{\mathcal B},\qquad
 \max_i\|P_Ng\|_{X_{\Psi_i}}\leq C\|g\|_{\mathcal B}.
\end{equation}
For fixed $A_0>0$, the collar terms follow by dividing the last estimate
in \eqref{old:weighted-poisson} by $A_0b$ or $A_0$ and further decreasing
$Q$. These choices impose no additional restriction on $\nu$.

\begin{proposition}\label{old:product-dtn}
Let $\mathcal D$ act on spherical degree $\ell$ as multiplication by
$-(\ell+2a)$. For sufficiently small universal $\nu$,
\begin{equation}\label{old:dtn-comparison}
 g\longmapsto(\rho_i\partial_{R_i}P_Ng|_{\Gamma_i})_i
  =\operatorname{diag}(\mathcal D)+O(\nu^{2a})+\mathcal R_{\rm old}
   :\mathcal B\longrightarrow\mathcal B_1.
\end{equation}
For each fixed old problem and each $\varepsilon>0$, the operator
$\mathcal R_{\rm old}$ has norm at most $\varepsilon$ if the largest
original radius and $Q$ are sufficiently small. Those choices may
depend on the old problem, $\nu$, and $\varepsilon$.
\end{proposition}

\begin{proof}
For an isolated sphere, the exterior harmonic Poisson extension
$P_i^{\rm harm}$, vanishing at infinity, has degree-$\ell$ radial factor
$(\rho_i/R_i)^{\ell+2a}$. Let $K$ be its off-diagonal trace matrix and
$M$ its off-diagonal normalized normal matrix. The trace on the $j$th
sphere is bounded in scaled $C^{2,\theta}$ by
$$
 C B_i\rho_i^{2a}|y_j-y_i|^{-2a}
       \|g_i/B_i\|_{C^{2,\theta}},
$$
since each scaled derivative contributes a factor
$\rho_j/|y_j-y_i|\leq\nu$. Summing by rows using \eqref{old:rows} gives
$$
 \|K\|_{\mathcal B\to\mathcal B}
 +\|M\|_{\mathcal B\to\mathcal B_1}\leq C\nu^{2a}.
$$
For small universal $\nu$, a Neumann series constructs the simultaneous
harmonic lift $P_Lg$. With $\mathcal D_0=\operatorname{diag}(\mathcal D)$,
its normal operator is
$$
 (\mathcal D_0+M)(I+K)^{-1}
 =\mathcal D_0+(M-\mathcal D_0K)(I+K)^{-1}
 =\mathcal D_0+O(\nu^{2a})
 \quad\text{in }\mathcal L(\mathcal B,\mathcal B_1).
$$
This includes reflected contributions to the diagonal.

Put $b_j=\|g_j\|_{C^{2,\theta}}$ and
$$
 E_g(x)=\sum_j\frac{b_j\rho_j^{2a}}
                    {(\rho_j^2+|x-y_j|^2)^a}.
$$
The harmonic monopoles dominate $|g|$ on every new sphere, so the
maximum principle for the harmonic lift gives $|P_Lg|\leq CE_g$.
Each summand is a positive Newton potential, and
$$
 \int_D E_g\,\dd x\leq C_{\rm old}q_g.
$$
On every doubled $i$th collar, \eqref{old:rows} gives
$E_g\leq CB_i\|g\|_{\mathcal B}$; harmonic boundary estimates give the
same bound for the scaled jets of $P_Lg$.

The difference $z=P_Ng-P_Lg$ has zero new trace and satisfies
$L_0z=-V_0P_Lg$, with old interface forcing
$T(P_Lg|_\Sigma)-\partial_{\mathbf n}P_Lg$.
The forcing and the source away from a fixed new neighborhood have
fixed norms $O(C_{\rm old}q_g)$, hence responses of that size by the old
interface estimate and \eqref{old:regular-source}.
For the localized source, use \eqref{old:green-kernel} only when both
source and observation lie in that neighborhood. The 3G estimate
\eqref{pot:3Goperator} gives
$$
 \mathcal I_2(\mathbf1_{B(x,L)}E_g)(x)\leq CL^2E_g(x).
$$
Splitting the localized kernel integral at $B(x,L)$, and using the
mass bound for the far part and bounded kernel remainder, yields
$$
 |z(x)|\leq C_{\rm old}L^2E_g(x)+C_{\rm old}(L)q_g
$$
on all doubled new collars. With $\nu$ fixed, first choose $L$ small
and then $Q$ small, using
$q_g\leq C\nu^{2a}Q\|g\|_{\mathcal B}$ and $B_i\geq m_{\rm old}$.
This makes $|z|/B_i$ arbitrarily small simultaneously on every collar.
Finally decrease the largest original radius so that
$\rho_i^2V_0P_Lg/B_i$ is small in scaled $C^{0,\theta}$.
Zero-trace boundary Schauder estimates give
$$
 \|\rho_i\partial_{R_i}z/B_i\|_{C^{1,\theta}}
       \leq\varepsilon\|g\|_{\mathcal B},
$$
proving \eqref{old:dtn-comparison}. All constants in the harmonic
comparison and the rescaled collar estimates are independent of the
number of holes and their radius ratios; the old-dependent choices
are made only after $\nu$ is fixed.
\end{proof}

The radius and coefficient conditions are compatible with the scales of
Theorem~\ref{thm:insertion}. In that notation,
$$
 c_N=\sum_i(\nu r_i)^{2a}
 \leq C\nu^{2a}\eta^{-a}r_{\max}^{a}Q.
$$
Consequently all old inverse, regular-source, and interface errors
can be made small after the universal gluing scale has been fixed.

\section{The nonlinear exterior problem and matching}
\label{sec:nonlinear-matching}

\subsection{The nonlinear problem on the punctured core}\label{sec:exterior}
We solve the equation on the perforated core using the inverse from
Proposition~\ref{old:inverse}. Newton potentials control the singular
part of its kernel; the bounded remainder is estimated through source
masses. Constants $C$ depend only on the fixed data of
Theorem~\ref{thm:insertion}, and may additionally depend on a fixed
parameter-ball radius $R$ when written $C_R$. Constants $C_0$ may depend
on the fixed $\eta$, the old solution, the core, and the chosen neighborhoods. Throughout
this section, $o_0(1)$ denotes an error that tends to zero as the
maximal radius and total Newton coefficient tend to zero, with the old data and
$\eta$ fixed. These errors are uniform in the number of new punctures
and their radius ratios.

For the gluing scales in \eqref{eq:scales}, write
$$
 \nu=\eta^{2/(n+2)},\qquad
 s=\eta/\nu,\qquad
 \delta=s^2=\nu^n,\qquad
 k=\nu^2+\delta.
$$
Thus $\rho_i=\nu r_i$ and $\lambda_i/\rho_i=s$.

\subsubsection{Frozen local amplitudes and conformally transformed bubbles}
Define
\begin{equation}\label{eq:frozen-background}
 B_i=U_0(y_i)+\sum_{j\ne i}b_{\lambda_j}(y_i-y_j),\qquad
 \mathbf a_i^0=\frac{\nabla(U_0+\sum_{j\ne i}b_{\lambda_j,y_j})(y_i)}{2aB_i}.
\end{equation}
After reducing the maximal radius, Lemma~\ref{pot:background}
gives \eqref{pot:backgroundsize}--\eqref{pot:backgroundconsequences}.
The jet estimate also holds on $B(y_i,cr_i)$ for a fixed small
$c>0$, since disjointness keeps the other centers at comparable
distances and the old relative jets are bounded on the fixed
compact neighborhood.

The parameters are
\begin{equation}\label{eq:parameter-space}
 \mathbf p=((b_i,\mathbf c_i,\psi_i))_i\in
 \cE_N=\prod_i\bigl(\R\times\R^n\times
 \Pi_{\mathrm{hi}}C^{2,\theta}(\Sph^{n-1})\bigr),
\end{equation}
with the maximum of the usual component norms. Work on $\norm{\mathbf p}\leq Rk$.
In the $i$th inner ball, the constructions of Section~\ref{in:section}
are applied in the local coordinate $x-y_i$.
Set $B=B_i e^{b_i}$, $\mathbf a_i=\mathbf a_i^0+\mathbf c_i/\rho_i$,
and use the inner solution $I_i(\mathbf p)$ of Proposition~\ref{in:family}, with
high data $\phi_i=B_i\psi_i$.
In the notation of Proposition~\ref{in:family}, these choices mean
$B_*=B_i$, $\boldsymbol z=\mathbf a_i$, and
$\boldsymbol c=\rho_i\mathbf a_i=\rho_i\mathbf a_i^0+\mathbf c_i$.
The same special conformal transformation gives the bubble
\begin{equation}\label{eq:transformed-bubble}
 b_i^{\mathbf p}=b_{\lambda_i'}(\,\cdot-y_i-\xi_i),\quad
 C_i=1+\lambda_i^2|\mathbf a_i|^2,\quad
 \lambda_i'=\lambda_i/C_i,\quad
 \xi_i=\lambda_i^2\mathbf a_i/C_i.
\end{equation}
Put
\begin{equation}\label{eq:SRV}
 S_{\mathbf p}=U_0+\sum_i b_i^{\mathbf p},\quad
 \mathcal R_{\mathbf p}=S_{\mathbf p}^p-U_0^p-\sum_i(b_i^{\mathbf p})^p,\quad
 V_{\mathbf p}=\kappa p(S_{\mathbf p}^q-U_0^q).
\end{equation}
Set $g_i=I_i-S_{\mathbf p}$ on
$\Gamma_i=\partial B(y_i,\rho_i)$. The transformed bubble associated
with $y_i$ cancels exactly in this difference:
\begin{equation}\label{eq:own-cancellation}
 g_i=\mathcal T_{\mathbf a_i}(U_B-b_{\lambda_i})
       +H_{B_i\psi_i}+W_i
       -\bigl(U_0+\sum_{j\ne i}b_j^{\mathbf p}\bigr)
       \quad\hbox{on }\Gamma_i.
\end{equation}
For $\norm{g}_{\cB}:=\max_i B_i^{-1}
\norm{g_i}_{C^{2,\theta}(\Sph^{n-1})}$ this proves
\begin{equation}\label{eq:g-estimates}
 \norm{g}_{\cB}\leq C_Rk,\qquad
 \norm{D_{\mathbf p}g}_{\cE_N\to\cB}\leq C_R.
\end{equation}
This cancellation is necessary: the two terms involving the own
bubble need not be bounded separately after division by $B_i$.

Outside the cutoff balls, unit parameter variations satisfy
\begin{equation}\label{eq:bubble-tangents}
 |\dot b_i^{\mathbf p}|\leq C_R\delta b_i^{\mathbf p},\quad
 |\dot{\mathcal R}_{\mathbf p}|\leq C_R\delta\mathcal R_{\mathbf p},\quad
 |\dot V_{\mathbf p}|\leq C_R\delta V_{\mathbf p}.
\end{equation}
The same bounds hold with scaled local derivatives of fixed order:
the displacement derivative in \eqref{eq:transformed-bubble} is
$O(\lambda_i^2/\rho_i)$ at observation distance at least $\rho_i$, and
the scale derivative uses $\lambda_i^2|\mathbf a_i|/\rho_i=O_R(\delta)$.
The positive Euler identities below give the remaining bounds.

\subsubsection{Weighted H\"older spaces}
Use the parameter-independent weights in \eqref{eq:weights},
and write $b_\Phi=b$ and $b_{\Psi_i}=1$. The choice $b=Q^{1/n}$
ensures that the height--mass response and bounded kernel
remainders become small after division by $bU_0$.
Choose fixed neighborhoods
$K\Subset U_-\Subset U_0'\Subset U_1\Subset U_2\Subset D$.
All original balls lie in $U_-$; choose $\chi=1$ on $U_0'$ with
$\supp\chi\subset U_1$.

Near $\Gamma_i$ use boundary patches of scale $\rho_i$, in a collar
of fixed scaled width. Elsewhere, away from fixed old collars, use
interior patches of scale
\begin{equation}\label{eq:patch-scale}
 \ell(x)=\min\{r_0,\min_j|x-y_j|\}.
\end{equation}
Patch radii are a small fixed multiple of $\ell$. Choose old-dependent
$0<r_0\leq1$ so enlarged patches stay in their stated regions, and require
$\max_i r_i\leq r_0$, hence $\rho_i\leq r_0$. With $\nu\ll1$, enlarged
boundary patches avoid other cutoff balls. Bounded relative jets of $W$
make its patch infimum comparable with its central value. Define
\begin{equation}\label{eq:space-X}
 \norm{v}_{X_W}=\sup\frac{|v|}{W}
 +\sup_P\frac{\norm{v}_{C^{2,\theta}_{\mathrm{sc}}(P)}}{\inf_P W}
 +\frac{\norm{v}_{C^{2,\theta}(\text{old collars})}}{A_0 b_W}.
\end{equation}
Scaled norms multiply the $j$th derivative by the patch scale to the
$j$th power and the H\"older seminorm by the corresponding power
plus $\theta$. The constant $A_0$ is fixed after the old data.
Let $X_W^0$ mean zero full trace on every $\Gamma_i$, and put
$X^0= X_\Phi^0\cap\bigcap_i X_{\Psi_i}^0$, with the maximum norm.
These are Banach spaces for each finite configuration, with weights
independent of the gluing parameter vector.

Extract the original ball containing an interior point, if any, and
apply packing to the remaining family:
\begin{equation}\label{eq:h-local}
 h\leq Cs^a\ell^{-a},\qquad h\leq CQ\ell^{-2a}.
\end{equation}
The first bound for the own term uses $|x-y_i|\geq\rho_i=\nu r_i$.
For every other term $r_j\leq|x-y_j|$, and shell summation applies.
The same proof applies in a new boundary patch with $\ell=\rho_i$.
In particular,
\begin{equation}\label{eq:V-local}
 \ell^2\norm{V_{\mathbf p}}_{C^{0,\theta}_{\mathrm{sc}}}
 \leq C\delta+o_0(1).
\end{equation}
Here $o_0(1)\to0$ as $\max r_i$ and $Q$ decrease with the old data and
$\eta$ fixed. For $q\leq1$ use
$V_{\mathbf p}\leq Ch^q$. For $q>1$, necessarily $a<2$, split at $h=U_0$.
In the low region use $V_{\mathbf p}\leq C_0h$ and optimize the two bounds in
\eqref{eq:h-local}; in the high region use $Ch^q$. This also proves
the scaled H\"older version by smooth positive ratio identities.

\subsubsection{Estimates for the residual and the potential}
For $p<2$, we use the positive integral formula
\begin{equation}\label{eq:euler-positive}
 D_p(z,w):=(z+w)^p-z^p-w^p
 =p(p-1)zw\int_0^1\int_0^1(tz+sw)^{p-2}\,\dd t\,\dd s.
\end{equation}
Every fixed-order Euler derivative in $z,w$ is bounded in absolute value
by $C D_p(z,w)$:
differentiation produces ratios between zero and one against the
positive integrand. The expression is smooth for $z,w>0$, with an
integrable singularity at $(t,s)=(0,0)$. Telescoping $\mathcal R_{\mathbf p}$ into the positive terms
$D_p(U_0+\sum_{j<i}b_j^{\mathbf p},b_i^{\mathbf p})$ therefore yields
\begin{equation}\label{eq:R-jets}
 \norm{\mathcal R_{\mathbf p}}_{C^{m,\theta}_{\mathrm{sc}}(P)}
 \leq C\mathcal R_{\mathbf p}(x_P),\quad
 \norm{\dot{\mathcal R}_{\mathbf p}}_{C^{m,\theta}_{\mathrm{sc}}(P)}
 \leq C\delta\mathcal R_{\mathbf p}(x_P).
\end{equation}
For every $i$, splitting at $H_i=U_0$ and
$h_i=U_0/2$ gives
\begin{equation}\label{eq:refined-source}
 \mathcal R_{\mathbf p}+V_{\mathbf p}\Psi_i
 \leq C\{h_i^q\Psi_i+H_i^p+
             U_0^qh\,\mathbf1_{\{h\leq C U_0\}}\}.
\end{equation}
Use comparability of the transformed bubbles with $h_i$. If
$H_i\geq U_0$, the terms containing $H_i^qU_0$ are bounded by $H_i^p$;
if $H_i<U_0$ and $h_i\geq U_0/2$, they are bounded by $h_i^qU_0$;
in the remaining low region they are bounded by $C U_0^qh$. The splitting
applies to both $q<1$ and $q\geq1$.

The first two terms in \eqref{eq:refined-source} have Newton potentials
bounded by $C\delta\Psi_i$, by the single-term 3G inequality and the
packed potential estimates. The last term has bounded height $C_0$ and
mass $C_0Q$. For $f\geq0$ with $\|f\|_\infty\leq L$ and
$\|f\|_1\leq m$, the standard near--far potential estimate
\cite{Hedberg1972} gives
\begin{equation}\label{eq:height-mass}
 \mathcal I_2 f\leq C L^{1-2/n}m^{2/n}.
\end{equation}
Set $\tau=Q^{2/n}=b^2$. Since $\min_D U_0>0$, its contribution
is at most $C_0\tau U_0=C_0b^2U_0$.

For the source masses, on $h\geq U_0$ each of $\mathcal R_{\mathbf p},V_{\mathbf p}U_0,V_{\mathbf p}h$ is at most $Ch^p$;
on $h<U_0$ each is at most $C_0h$. Equation~\eqref{pot:mass} gives $\int h^p\leq C\delta Q$.
Also $\int_Dh\leq C_0Q$. Consequently,
\begin{equation}\label{eq:source-masses}
 \int_D(\mathcal R_{\mathbf p}+V_{\mathbf p}U_0+V_{\mathbf p}h)\,\dd x\leq C_0Q.
\end{equation}
Use these masses rather than summing the larger majorants $h_i^qU_0$.

\begin{lemma}[Weighted inverse estimates]\label{ext:structural}
Let $G_N[F,\xi]$ denote the inverse with zero trace on the new spheres,
source $F$, and old interface forcing $\xi$ as in
\eqref{old:linear-boundary-problem}; write $G_NF=G_N[F,0]=G_N^0F$.
Choose $r_0$ and $A_0$ after fixing the old data and $\eta$.
After reducing the radius and coefficient bounds, for $W=\Phi$ and each
$\Psi_i$,
\begin{align}\label{eq:structural}
 \norm{G_N(V_{\mathbf p} f)}_{X_W}&\leq(C\delta+o_0(1))\norm{f}_{X_W},
 \\
 \norm{G_N\mathcal R_{\mathbf p}}_{X_W}&\leq C\delta.
 \nonumber
\end{align}
For every fixed smooth multiplier $m$,
$\|G_N(mf)\|_{X_W}\leq C_{0,m}\|f\|_{X_W}$.
For $\|\dot{\mathbf p}\|_{\mathcal E_N}\leq1$, the parameter derivatives satisfy
\begin{equation}\label{eq:structural-derivatives}
 \|G_N(D_{\mathbf p}V_{\mathbf p}[\dot{\mathbf p}]f)\|_{X_W}
       \leq C_R\delta^2\|f\|_{X_W},\qquad
 \|G_ND_{\mathbf p}\mathcal R_{\mathbf p}[\dot{\mathbf p}]\|_{X_W}
       \leq C_R\delta^2.
\end{equation}
The interface forcing obeys
\begin{equation}\label{eq:interface-weighted-bound}
 \|G_N[0,\xi]\|_{X_W}
       \leq C_0 b_W^{-1}\|\xi\|_{C^{1,\theta}(\Sigma)},
 \qquad b_\Phi=b,\quad b_{\Psi_i}=1.
\end{equation}
Here the old-dependent constants are uniform in the new configuration.
\end{lemma}
\begin{proof}
First use the localized kernel of Proposition~\ref{old:inverse} for source
and observation in $U_2$. For a $\Psi_i$ majorant apply
\eqref{eq:refined-source}. The singular part is at most
$C\delta\Psi_i+C_0b^2U_0$. The smooth part of the kernel is bounded
by \eqref{eq:source-masses}. For $\Phi$, split the multiplier source
into its $h$ and $bU_0$ parts. For the residual, the high--low
splitting used in \eqref{eq:source-masses} gives on $D_N$
$$
 \mathcal R_{\mathbf p}\leq Ch^p+C_0h\mathbf1_{\{h<U_0\}},\qquad
 \mathcal I_2(\mathbf1_{D_N}\mathcal R_{\mathbf p})
       \leq C\delta h+C_0b^2U_0.
$$
The residual source in the Newton potential is extended by zero
outside $D_N$.
The low-region term has bounded height and mass $O(C_0Q)$;
\eqref{eq:height-mass} and $\min_DU_0>0$ give its stated response.
Dividing the residual response by $\Phi=h+bU_0$ gives
$C\delta+C_0b$. Dividing the height--mass response
$C_0b^2U_0$ by the positive background term $bU_0$ gives $C_0b$; the
$bU_0$ multiplier part contributes at most $C\delta+C_0b^2$.
The smooth source mass contributes $C_0Q/b$. Thus both the multiplier
and the residual have coarse value bound at most
$$
 C\delta+C_0\left(b+b^2+\frac{Q}{b}\right).
$$
The old-dependent terms vanish with $Q$, since $Q/b=Q^{1-1/n}$.

Apply the kernel only to $\chi F$. Outside $U_2$, separation bounds the
response by $C_0\int|F|$.
Put $F_V=V_{\mathbf p}f$ and $F_R=\mathcal R_{\mathbf p}$. On the fixed support of
$1-\chi$, all bubbles and their ordinary jets are $O(C_0Q)$.
Require $b\leq1$ and $Q\leq b$. There $W=O(C_0b_W)$,
and hence
$$
 \norm{(1-\chi)F_V}_{C^{0,\theta}}
       \leq C_0Qb_W\norm{f}_{X_W},\qquad
 \norm{(1-\chi)F_R}_{C^{0,\theta}}\leq C_0Q.
$$
These products extend by zero across all the new balls.
Apply \eqref{old:regular-source} to these two regular sources.
The old inverse bounds their no-hole responses by
$C_0Qb_W\norm{f}_{X_W}$ and $C_0Q$, respectively, in fixed
$C^{2,\theta}$ norms on the core and old collars.
Since $B_i\geq m_{\rm old}>0$, the new-sphere traces have the same
bounds in $\mathcal B$. Applying \eqref{old:poisson-X} with the
particular weight $W=\Phi$ or $\Psi_i$, and dividing the old-collar
estimates by $A_0b_W$, gives
\begin{align*}
 \norm{G_N((1-\chi)V_{\mathbf p}f)}_{X_W}
       &\leq C_0Q\norm{f}_{X_W},\\
 \norm{G_N((1-\chi)\mathcal R_{\mathbf p})}_{X_W}
       &\leq C_0Q/b_W.
\end{align*}
For prescribed $\varepsilon_0>0$, reduce $Q$ so that
\begin{equation}\label{eq:regular-budget}
 C_0\left(b+b^2+\frac{Q}{b}\right)
       \leq\varepsilon_0\delta.
\end{equation}
Since $b=Q^{1/n}\to0$ and $Q/b=Q^{1-1/n}\to0$, this closes
\eqref{eq:structural}, including the old-collar terms in the weighted norms. In particular,
the coefficient of $\delta$ in these estimates is universal; the
old-dependent constants are absorbed by reducing $Q$ after $\eta$
has been fixed.

For the scaled local norm apply boundary or interior Schauder
estimates after rescaling the patches. After the old data and $\eta$
are fixed, choose $r_0$ small enough that, on every new boundary
or interior patch $P$ of scale $\ell_P$,
$$
 \ell_P^2\norm{\kappa pU_0^q}_{C^{0,\theta}_{\mathrm{sc}}(P)}
       \leq\delta,
$$
and the finite-order scaled relative jets of $U_0$ and the weights
$W$ needed here have universal bounds. These choices are possible
because $U_0$ is positive and smooth on the fixed core neighborhood;
the regularized potentials $h_j$ have uniformly bounded scaled relative jets.
The new potential is bounded by \eqref{eq:V-local}. For the residual source H\"older norm, use $\mathcal R_{\mathbf p}/W\leq C S_{\mathbf p}^q$,
\eqref{eq:R-jets}, and \eqref{eq:h-local}. At fixed old collars apply the separated
mass estimate and \eqref{old:regular-source}. The ordinary old-collar estimates are divided by $A_0b_W$.
Fix $A_0\geq1$ before imposing the final radius and coefficient bounds;
then reduce $Q$ so these terms obey the same inequalities.

For \eqref{eq:interface-weighted-bound}, let $w_0$ be the
smooth no-hole solution with old-interface forcing $\xi$.
The identity \eqref{old:regular-source} applies equally to this
forcing. Since $\|w_0\|_{C^{2,\theta}(D)}+
\|w_0|_{\Gamma_N}\|_{\mathcal B}\le C_0\|\xi\|_{C^{1,\theta}}$,
\eqref{old:poisson-X}, the lower bound $W\ge b_WU_0$, and the
old-collar estimates prove the claim.

Write $\mathcal I_{2,D}f=\mathcal I_2(\mathbf1_D f)$ for the Newton
potential of $f$ extended by zero outside $D$.
Finally, $U_0$ and all $h_i$ are positive Newton potentials. The 3G
inequality therefore gives $\mathcal I_{2,D}W\leq C\diam(D)^2W$. Its mass on $D$ is bounded by an old-dependent multiple of $b_W+Q$.
The same patch and collar argument proves bounded multiplication.
By \eqref{eq:bubble-tangents} the residual derivative has an
additional factor $\delta$: its regular coarse response is at most
$C_0\delta Q/b\leq\varepsilon_0\delta^2$, and its refined
response is at most $C_0\delta Q$. The same additional factor applies to the differentiated multiplier,
proving \eqref{eq:structural-derivatives}. No parameter-dependent
weights or domains are differentiated.
\end{proof}

\subsubsection{The exterior solution family}
Write $P=P_Ng$ and seek $E_{\mathbf p}=S_{\mathbf p}+P+v$, $v\in X^0$.
By Lemma~\ref{old:poisson},
\begin{align}\label{eq:P-estimates}
 \norm{P}_{X_\Phi}&\leq Ck\nu^{2a}=O(\delta),&
 \max_i\norm{P}_{X_{\Psi_i}}&\leq Ck,
 \\
 \norm{\dot P}_{X_\Phi}&\leq C\nu^{2a},&
 \max_i\norm{\dot P}_{X_{\Psi_i}}&\leq C.
 \nonumber
\end{align}
Here $k\nu^{2a}=O(\nu^{2+2a})=O(\delta)$.
The fixed old observations of $P,\dot P$ are
$O(C_0k\nu^{2a}Q),O(C_0\nu^{2a}Q)$.

Let $\mathcal N_{\mathrm{old}}(\gamma)=\mathcal N_{\mathrm{old}}(0)+T\gamma+\mathcal E(\gamma)$
be the exact graph of Proposition~\ref{old:graph}, with core normal on
both sides. For $w=P+v$ set
\begin{equation}\label{eq:nonlinear-N}
 N_{\mathbf p}(w)=\kappa\bigl((S_{\mathbf p}+w)^p-S_{\mathbf p}^p-pS_{\mathbf p}^q w\bigr),\qquad
 \gamma=(\sum_i b_i^{\mathbf p}+P+v)|_{\Sigma}.
\end{equation}
The equation and the old boundary relation are exactly the fixed point
\begin{equation}\label{eq:exterior-fixedpoint}
 v=G_N\left[-\kappa\mathcal R_{\mathbf p}-V_{\mathbf p}(P+v)-N_{\mathbf p}(P+v),
      \ T\sum_i b_i^{\mathbf p}-\partial_{\mathbf n}\sum_i b_i^{\mathbf p}+\mathcal E(\gamma)\right].
\end{equation}
The notation $G_N[F,\xi]$ includes old interface forcing $\xi$.
The affine interface term accounts for the bubbles' old boundary discrepancy.

\begin{lemma}[Exterior existence and Cauchy estimates]\label{ext:family}
For a sufficiently small universal $\eta$ and sufficiently small
old-dependent radius and coefficient bounds, \eqref{eq:exterior-fixedpoint}
has a unique solution in a ball $C\delta$ of $X^0$, jointly $C^1$ in
$\mathbf p$. Uniformly on $\norm{\mathbf p}\leq Rk$,
\begin{align}\label{eq:v-estimates}
 \norm{v}_{X_\Phi}+\max_i\norm{v}_{X_{\Psi_i}}&\leq C\delta,
 \\
 \norm{\dot v}_{X_\Phi}&\leq C(\delta\nu^{2a}+\delta^2),&
 \max_i\norm{\dot v}_{X_{\Psi_i}}&\leq C\delta.
 \nonumber
\end{align}
Moreover
\begin{equation}\label{eq:exterior-Cauchy}
 \max_i\left(\norm{\rho_i\partial_{R_i}v/B_i}_{C^{1,\theta}}
 +\norm{D_{\mathbf p}(\rho_i\partial_{R_i}v/B_i)}_{\cE_N\to C^{1,\theta}}\right)
 \leq C\delta.
\end{equation}
The exterior Jacobi problem has zero kernel among fields with zero
full trace on the new spheres, old radial coefficient vector in
$\mathcal W'$, and physical remainder decaying at rate $d_-$ on
every eliminated old tail.
\end{lemma}
\begin{proof}
On the intersection ball, $|w|\leq C\delta\Phi$ and
$b$ is small, so $|w|/S_{\mathbf p}<1/2$. Smooth Taylor formula gives
\begin{equation}\label{eq:N-multiplier}
 |D_wN_{\mathbf p}(w)|\leq C|w|S_{\mathbf p}^{q-1}
 \leq C\delta V_{\mathbf p}+C_0\delta b.
\end{equation}
For the last inequality split at $h=U_0$; equivalently use
$hS_{\mathbf p}^{q-1}\leq C(S_{\mathbf p}^q-U_0^q)$ and
$U_0S_{\mathbf p}^{q-1}\leq C((S_{\mathbf p}^q-U_0^q)+U_0^q)$.
Lemma~\ref{ext:structural} controls the first term. The second term
is $b$ times a fixed old smooth multiplier, so its inverse response
has norm at most $C_0\delta b$. Thus it imposes a condition on $Q$,
without imposing an old-dependent restriction on $\eta$.
For the scaled H\"older multiplier bound, on a patch $P$ write
$D_wN=\kappa pS_{\mathbf p}^q((1+w/S_{\mathbf p})^q-1)$. The relative jets of $S_{\mathbf p}$ and
$\Phi$, and the coarse local jets of $w$, give
$$
 \norm{D_wN}_{C^{0,\theta}_{\mathrm{sc}}(P)}
 \leq C\delta\Phi(x_P)S_{\mathbf p}(x_P)^{q-1}
 \leq C\delta\bigl(V_{\mathbf p}(x_P)+b U_0(x_P)^q\bigr).
$$
Multiplying by the local jets of $f$ and using the same positive
envelopes as in Lemma~\ref{ext:structural} therefore gives
$\norm{G_N(D_wN f)}_{X_W}
\leq[C\delta(C\delta+o_0(1))+C_0\delta b]\norm{f}_{X_W}$.
On fixed old collars the extra constants carry $b$, hence shrink with $Q$.
Also
$\norm{\gamma}_{C^{2,\theta}}\leq
C_0(Q+A_0\delta b)$. Reducing $Q$ makes the old graph derivative small. Equation~\eqref{eq:interface-weighted-bound} yields, for the nonlinear
boundary contribution,
$$
 \|G_N[0,D\mathcal E(\gamma)(f|_\Sigma)]\|_{X_W}
 \leq C_0A_0\|\gamma\|_{C^{2,\theta}(\Sigma)}\|f\|_{X_W}.
$$
The affine interface forcing has size $O(C_0Q)$, and hence coarse
inverse norm $O(C_0Q/b)$. The quadratic graph term has coarse inverse
norm at most $C_0(Q+A_0\delta b)^2/b$. With $\eta,A_0$ fixed, both
are made as small as required by reducing $Q$. The refined estimates
have no factor $1/b$.

The preceding bounds show that the fixed-point map preserves a
$C\delta$ ball in $X^0$ and has derivative in $v$ of norm less than
$1/4$, so Banach's theorem gives a unique solution. The intersection
ball controls the old trace; at the solution the same contraction
bound holds on every individual $X_W$.

For each fixed finite configuration, the domain and weights are fixed,
and the weighted norms are equivalent to ordinary H\"older norms,
with configuration-dependent constants. The $C^1$ inner data from
Proposition~\ref{in:family}, the $C^2$ old graph, and the smooth power
nonlinearity in the positive neighborhood permit the implicit function
theorem, giving $C^1$ dependence on $\mathbf p$ in $X^0$. The norm equivalences
provide no uniform constants; the uniform derivative estimates follow
by differentiating \eqref{eq:exterior-fixedpoint}.
Keep $\dot P$ explicit with the orders in \eqref{eq:P-estimates}; it is
not of order $\delta$ in every refined space.
Its force $V_{\mathbf p}\dot P$ has coarse order $\delta\nu^{2a}$ and refined
order $\delta$. The derivatives of $\mathcal R_{\mathbf p},V_{\mathbf p}$ have the
extra factor $\delta$ in \eqref{eq:bubble-tangents}.

To differentiate the Taylor remainder, write
$N(S,w)=w^2 S^{q-1}F(w/S)$ and
$N_S(S,w)=w^2 S^{q-2}F_1(w/S)$, with smooth $F,F_1$ for $|w/S|<1/2$.
The scaled jets of $w$ are bounded simultaneously by
$\mathfrak A_c=C\delta\Phi$ and $\mathfrak A_i=Ck\Psi_i$; hence for $\mathfrak A_W=\mathfrak A_c$ or $\mathfrak A_i$,
$$
 \norm{N_S\dot S}_{C^{0,\theta}_{\mathrm{sc}}}
 \leq C\delta h\mathfrak A_c\mathfrak A_W S^{q-2}
 \leq C\delta^2 V_{\mathbf p}\mathfrak A_W.
$$
These positive amplitude bounds also control the scaled jets.
For the old derivative terms put $\zeta_{\rm old}=Q/b$.
A unit parameter tangent has affine old forcing of ordinary size
$C_0\delta Q$, whose coarse inverse response is
$C_0\delta\zeta_{\rm old}\leq\varepsilon_0\delta^2$ by
\eqref{eq:regular-budget}. Moreover,
$$
 \norm{\gamma}_{C^{2,\theta}}\leq C_0(Q+A_0\delta b),\qquad
 \norm{\textstyle\sum_i\dot b_i^{\mathbf p}+\dot P}_{\rm old}
       \leq C_0(\delta+\nu^{2a})Q.
$$
The part of $D\mathcal E(\gamma)$ containing these explicit
tangents consequently has coarse response at most
$$
 C_0(Q+A_0\delta b)(\delta+\nu^{2a})\zeta_{\rm old}.
$$
After \eqref{eq:regular-budget}, further reducing $Q$ makes this
at most $\varepsilon_0(\delta^2+\delta\nu^{2a})$. The part containing $\dot v$ is included in the linearized
fixed-point operator; its norm
is at most $C_0A_0(Q+A_0\delta b)$, which can be made less than
$1/16$. Also require $C_0b\leq\varepsilon_0$ for the old bounded
multiplier term. All these conditions are compatible with fixed
$\eta,A_0,r_0$, since $b,\zeta_{\rm old}\to0$. In the refined
spaces the factors $1/b$ are absent, and the same explicit sources
are $O(\delta)$. Inverting this linearization in each $X_W$
proves \eqref{eq:v-estimates}. Boundary Schauder estimates with zero
new trace prove \eqref{eq:exterior-Cauchy}.

Finally let $J$ be a current exterior Jacobi field with radial
coefficient vector in $\mathcal W'$, physical remainder decaying at rate
$d_-$ on every eliminated old tail, and zero full trace on the new
spheres. For fixed $N$ it has finite $X_W$ norms on the core.
The lower-weight tail uniqueness in Proposition~\ref{old:graph} gives the
exact linearized old boundary relation. It follows that
$J=D_v\mathcal T(v)J$, where $\mathcal T$ is the right side of
\eqref{eq:exterior-fixedpoint}. The contraction estimate forces $J=0$.
\end{proof}
\subsection{Matching and nondegeneracy}\label{sec:matching}
We match the interior and exterior Cauchy data, following the
Dirichlet-to-Neumann method of \cite[Section~7]{MP1999} and
\cite[Sections~3--5]{Santos2010}.
Their Dirichlet traces already agree. On both sides of each new sphere,
$\partial_{R_i}$ points away from $y_i$, opposite to the exterior outward normal.

\subsubsection{The matching operator}
Define
\begin{equation}\label{eq:mismatch}
 \cM_i(\mathbf p)=\frac{\rho_i}{B_i}
     (\partial_{R_i} I_i-\partial_{R_i} E_{\mathbf p})|_{\Gamma_i}.
\end{equation}
The target is the maximum product of $C^{1,\theta}(\Sph^{n-1})$.
The leading operator from \eqref{eq:parameter-space} is
\begin{equation}\label{eq:leading-matching}
 L_i(b_i,\mathbf c_i,\psi_i)
 =2ab_i+2an\,\mathbf c_i\cdot\omega
       +\sum_{\ell\geq2}(2\ell+2a)(\psi_i)_\ell.
\end{equation}
The subscript $\ell$ denotes the full degree-$\ell$ spherical harmonic
component. Identify the parameter triple with the boundary function
$$
 f_i=b_i+2a\,\mathbf c_i\cdot\omega+\psi_i
       \in C^{2,\theta}(\Sph^{n-1}).
$$
This identification and its inverse are bounded with constants depending
only on $n$ and $\theta$. If $\mathcal D_{\mathrm{in}}$ is the Dirichlet-to-Neumann
operator for regular harmonic functions on $B_1$, then
\begin{equation}\label{eq:matching-robin}
 L_i(b_i,\mathbf c_i,\psi_i)
   =2(\mathcal D_{\mathrm{in}}+a)f_i
   =2\sqrt{-\Delta_{\Sph^{n-1}}+a^2}\,f_i.
\end{equation}
Indeed $\mathcal D_{\mathrm{in}}$ has eigenvalue $\ell$ on degree $\ell$, and
$2(a+1)(2a)=2an$ on the first-degree block.

\begin{lemma}\label{match:linear}
$L$ is an isomorphism from $\cE_N$ to the target product space.
Its inverse norm is independent of $N$ and the radii.
\end{lemma}
\begin{proof}
The positive Robin operator $\mathcal D_{\mathrm{in}}+a$ has a
bounded inverse $C^{1,\theta}\to C^{2,\theta}$ by the standard
Robin Schauder estimate \cite[Chapter~6]{GT2001}.
Equation~\eqref{eq:matching-robin} and the maximum product norm
give the assertion.
\end{proof}

\begin{lemma}\label{match:estimate}
On $\norm{\mathbf p}\leq Rk$,
\begin{align}\label{eq:mismatch-estimates}
 \norm{\cM(0)}&\leq C_*k+o_0(1),
 \\
 \norm{D\cM-L}&\leq C_R(\nu+s^{2-\mu_3})+o_0(1),
 \nonumber
\end{align}
The baseline constant $C_*$ is chosen independently of $R$.
\end{lemma}
\begin{proof}
The own transformed bubbles cancel. The Cauchy-data formulas of
Section~\ref{in:section} give trace/flux pairs $(B,0)$ for a $b$
variation and $(2aB,2aB+O(\delta B))\dot{\mathbf c}\cdot\omega$
for a special conformal variation at zero. Combining these with
the exterior eigenvalues $-(\ell+2a)$ and the interior high-mode
eigenvalues $\ell$ gives the three blocks of
\eqref{eq:leading-matching}, with error $O(\delta)$.

At the baseline, Taylor-expand
$U_0+\sum_{j\ne i}b_{\lambda_j,y_j}$ at $y_i$.
Its constant term is $B_i$ and its gradient agrees with
$2aB_i\mathbf a_i^0$. The second-order error on the cutoff sphere
is $O(\nu^2B_i)$ by \eqref{pot:backgroundjets}. The change in
other bubbles caused by \eqref{eq:transformed-bubble} is
$O(\eta^2B_i)$ at the baseline and $O(\delta B_i)$ for unit
tangents. For $V_B$, the boundary data and radial equation bound the required
nonconstant jets and their first parameter derivatives on a two-sided
annulus by $O(\delta B)$.
Taylor expansion in $|\mathbf a_i|\rho_i=O(\nu+Rk)$ gives the
low-mode derivative error $O(\nu+\delta+Rk)$.

Proposition~\ref{in:family} bounds the normalized Cauchy-data
contribution of the inner correction and its parameter derivative by
$O_R(\nu+s^{2-\mu_3})$. Indeed, the fixed gap
$d=\mu_3-\mu_2>0$ gives
$$
 (1+\log(1/s))s^{2-\mu_2}
 =s^{2-\mu_3}(1+\log(1/s))s^d
 \leq C_d s^{2-\mu_3},
$$
while $\nu^aRk\leq2R\nu$ since $k\leq2\nu^2$ and $\nu\leq1$.
Proposition~\ref{old:product-dtn} gives the
full exterior Poisson flux as its diagonal harmonic operator plus
$O(\nu^{2a})+o_0(1)$. The derivatives of the boundary discrepancy $g_i$, after the own
bubble cancellation, are bounded in the $B_i$-normalized norm by
\eqref{eq:g-estimates}.
The zero-trace nonlinear correction has flux error $O(\delta)$ by
\eqref{eq:exterior-Cauchy}. Since $n\geq3$ and $\nu\leq1$,
$\delta=\nu^n\leq\nu$, $\nu^{2a}=\nu^{n-2}\leq\nu$, and
$Rk\leq2R\nu$. Thus all the universal derivative errors are absorbed
in $C_R(\nu+s^{2-\mu_3})$, proving the derivative estimate.

At $\mathbf p=0$ the high data and its correction vanish. The exterior
correction is bounded by its force times $(1-1/4)^{-1}$, independently
of the parameter-ball radius. Its flux and the Taylor errors therefore
give $C_*$ independently of $R$.
\end{proof}

Choose
$R=2\norm{L^{-1}}(C_*+1)$, and then choose $\eta_*$ universally so
the universal error in \eqref{eq:mismatch-estimates} is less than
$1/(4\norm{L^{-1}})$. The old-dependent radius and coefficient choices
make the remaining derivative error smaller than the same quantity
and the remaining baseline error at most $k$. The map $\mathbf p\mapsto\mathbf p-L^{-1}\cM(\mathbf p)$ contracts the ball $\norm{\mathbf p}\leq Rk$
into itself. At its fixed point, the positive interior and exterior
solutions have matching values and normal derivatives. Green's
identity gives the equation across every interface without a surface
measure, and interior elliptic regularity makes it smooth.

\subsubsection{Representation of decaying Jacobi fields}
The absence of a decaying kernel follows from matching only after
one proves that the inner parameter tangents exhaust all decaying
Jacobi fields. We establish that representation next.

\begin{lemma}\label{match:inner-kernel}
Let $I=\mathcal T_{\mathbf a}U_B+Z$ be an inner solution above, with
$Z\in C^{2,\theta}_{\mu_3}$. Every additive Jacobi field $J$ for $I$
that belongs to $C^{2,\theta}(\overline{B_\rho}\setminus B_{\rho'})$
for every $0<\rho'<\rho$, and whose relative factor $J/I$ is $O(r^d)$ at the puncture for some
$d>0$ is a linear combination of the special conformal and
high-boundary-data tangents of $I$, with $A$ and $B$ fixed. Constants in this representation may depend
on the fixed positive neck.
\end{lemma}
\begin{proof}
The assumed regularity gives a $C^{2,\theta}$ boundary trace.
Since $I\asymp r^{-a}$ for the fixed positive neck, the relative
bound implies $J=O(r^{-a+d})$. Annular Schauder estimates for the
Jacobi equation give the corresponding scaled derivative bounds. Write $U=\mathcal T_{\mathbf a}U_B$.
For a positive conformal factor $f$, write $L_f=\Delta+\kappa p f^q$
for its additive Jacobi operator, as in \eqref{in:inverse}.
Undo the special conformal transformation on a smaller concentric ball
in the transformed domain. The pullback fixes the origin, has comparable
radial variables and a smooth nonvanishing conformal factor, and
preserves the strong additive weights.
On this ball the transformed solution and Jacobi field satisfy
$$
 \widetilde I=U_B+\widetilde Z,
 \qquad \widetilde Z\in C^{2,\theta}_{\mu_3},
 \qquad \widetilde J=O(r^{-a+d}),
$$
with scaled derivatives. In the following argument $\varepsilon>0$
denotes the fixed necksize, and $d>0$ denotes the decay exponent of
the field. Since $U_B\asymp r^{-a}$,
Taylor expansion gives
$$
 (L_{\widetilde I}-L_{U_B})\widetilde J
               =O(r^{\mu_3+d-2}).
$$
Set $j=r^a\widetilde J$ and $t=-\log r$ (up to an additive constant).
The radial cylindrical Jacobi equation is
$$
 P_\varepsilon j=f,\qquad
 P_\varepsilon=\partial_t^2+\Delta_{\mathbb S^{n-1}}
                   -a^2+\kappa p v^q,
 \qquad f=O(e^{-(a+\mu_3+d)t}),
$$
where $v=r^aU_B$ is a phase translate of the Fowler profile.
Since $\kappa p=n(n+2)/4$, this is the cylindrical normalization of
\cite[Section~2.2]{KMPS1999}.
Initially $j\in \mathcal C_{d_0}^{2,\theta}$ on the cylinder for any
$0<d_0<\min\{d,1\}$.

Choose $\gamma=a+\mu_1$, so $\gamma>1$ and $\gamma<a+2$.
By Lemma~2 of \cite[Section~5]{MP1999}, the first high-mode indicial
exponent $\gamma_{\varepsilon,n+1}$ has a lower bound arbitrarily close
to $a+2$ for small necks. Reducing the small-neck threshold in terms
of $n$ and $\mu_1$ gives $\gamma_{\varepsilon,n+1}>\gamma$. Since
$a+\mu_3+d>\gamma$, Corollary~1 of
\cite[Section~2.2]{KMPS1999} applies and gives
$$
 j\in \mathcal C_{\gamma}^{2,\theta}\oplus E_{\varepsilon,1}.
$$
Here $E_{\varepsilon,1}$ is the span of the stable first spherical
Jacobi fields. In additive Euclidean variables these are precisely
the fixed-center special conformal generators,
$r^{1-a}(av-v')\omega_j$. Hence $\widetilde J$ is a combination of
these generators plus a field in $X_{\mu_1}$. The initial positive decay excludes the two neutral radial fields
and the growing first-degree fields. The weight conversion is
$j\in \mathcal C_{\gamma}^{2,\theta}\Longleftrightarrow
\widetilde J\in X_{\gamma-a}$; with $\gamma=a+\mu_1$, the remainder
therefore belongs to $X_{\mu_1}$. All norms are finite on the remaining compact annulus.

Return through the conformal transformation. The special conformal
transformations satisfy the local group law
$\mathcal T_{\mathbf a+\mathbf z}
=\mathcal T_{\mathbf a}\mathcal T_{\mathbf z}$.
Consequently their parameter tangents have precisely the stable
first-degree leading terms found above.
Proposition~\ref{in:family} places the derivatives of the correction
in $X_{\mu_3}$. Subtracting suitable special conformal tangents of
$I$ therefore leaves a field in $X_{\mu_1}$, with positive relative
decay. Choose $\dot\psi$ so that
$B_*\dot\psi$ equals the remaining high trace, and subtract the
corresponding high-boundary tangent of $I$, which belongs to $X_{\mu_3}$.
The remainder $J_0$ is a strong Jacobi field with zero high trace.
The local strong inverse gives
$$
 J_0=-G_U\bigl[(L_I-L_U)J_0\bigr].
$$
The operator on the right has norm less than one on $X_{\mu_1}$ by the
small relative potential estimate, so $J_0=0$.
\end{proof}

\begin{proposition}\label{match:nondegenerate}
The matched finite-ended metric is marked nondegenerate.
\end{proposition}
\begin{proof}
Let $J$ be a global marked Jacobi field. Its smoothness at the matching
spheres allows Lemma~\ref{match:inner-kernel} to represent it at every
new puncture as an inner tangent with $\dot b_i=0$. Passing from physical coefficients
to \eqref{eq:parameter-space} multiplies the special conformal
coefficient by $\rho_i$ and divides the high trace by $B_i$.
Since these finitely many factors are positive, this defines
$\dot{\mathbf p}\in\cE_N$. No uniform bound on the representation coefficients
is needed for this kernel argument.

At each old end the radial coefficients of a marked field are zero.
It therefore satisfies the current common gauge and the lower-weight
boundary graph of Proposition~\ref{old:graph}. Because $E_{\mathbf p}|_{\Gamma_i}=I_i(\mathbf p)|_{\Gamma_i}$ for all parameters,
$D_{\mathbf p}E_{\mathbf p}[\dot{\mathbf p}]$ has exactly the same full new trace as $J$.
The old radial coefficient vector of this exterior tangent belongs to
$\mathcal W'$, and its physical remainder has decay $d_-$ by
Proposition~\ref{old:graph}. Since the old radial coefficients of $J$
vanish, their difference belongs to the same current Jacobi space
and has zero full trace on the new spheres. Exterior uniqueness in
Lemma~\ref{ext:family} therefore makes it zero. The vanishing global
normal jump then says $D\cM\dot{\mathbf p}=0$. The bound in
\eqref{eq:mismatch-estimates} makes $D\cM$ invertible; hence
$\dot{\mathbf p}=0$, and all restrictions of $J$ vanish.
\end{proof}

\subsubsection{Order of choices and lower bounds}
\begin{proof}[Proof of Theorem~\ref{thm:insertion}]
The preceding matching argument and
Proposition~\ref{match:nondegenerate} give a smooth positive
marked nondegenerate solution. It remains to verify the uniform
order of choices, lower bound, and approximation statements.

Fix
the universal data, the three weights, and the parameter-ball radius
$R$ determined above. Choose $\eta_*$ from the local inverse thresholds, the high-mode
indicial bound in Lemma~\ref{match:inner-kernel},
inner Cauchy estimates, exterior contraction, and matching estimates.
Decrease $\eta_*$ so that $\eta_*\leq1$ and
$\eta_*^{2/(n+2)}<1/4$.
Fix $\eta\leq\eta_*$, then choose the old cuts, the regular infinity
cap, the two old weights, the common complement, and the old inverse
constants. Choose the fixed buffers, $r_0$, and $A_0$ next.
Reduce the maximal radius to meet the old jet bounds and
the localized Poisson comparison estimates, and impose
$10r_{\max}<\dist(K,\partial D_{\rm pre})$ on the prescribed core.
If $K$ is empty, only the trivial empty insertion is needed. Finally reduce $Q$ to
meet the old spectral, boundary-graph, and smoothing estimates.
The choices are compatible: $b=Q^{1/n}\to0$ and
$Q/b=Q^{1-1/n}\to0$, while the capacity estimate at the end of
Section~\ref{old:section} gives $\sum_i\rho_i^{2a}\to0$.
On the core the matched factor is $S_{\mathbf p}+P+v$ and
$|P+v|\leq C\delta(h+b U_0)$. The transformed bubbles are
uniformly comparable with $h_i$, and $S_{\mathbf p}\asymp U_0+h$.
For small universal $\delta$ and small old-dependent $b$ this
gives \eqref{eq:core-lower}. The own term gives
\eqref{eq:shell-lower}. The inner estimates imply complete,
positive, noncylindrical new ends; the old graph preserves the old
positive necks.

On a protected compact set disjoint from the new centers, all bubble
jets are $O(Q)$, and the correction bounded by $C\delta\Phi$ tends to zero with $b$.
Elliptic bootstrap gives convergence in every prescribed finite
smoothness norm. On an old tail the compact set occupies a finite
cylindrical interval, where the period-aligned family is continuous
in these norms. The compact Dirichlet implicit function theorem gives
the same conclusion on the regular cap.
\end{proof}

\section{Passage to a complete limit}
\label{sec:limit}

We now apply Theorem~\ref{thm:insertion} to construct a complete limit.
The argument uses summable
perturbations on compact sets, as in \cite[Theorem~4.1]{MPU1995}, and
cutoff functions whose transition regions remain disjoint throughout
the construction. The gluing parameter $\eta$ is fixed before the
induction; the permissible radii and total Newton coefficient may depend on the
solution at the preceding step.

\subsection{The induction}

We prove Theorem~\ref{thm:main}.

\begin{proof}[Proof of Theorem~\ref{thm:main}]
Choose $\dim_{\mathrm H}\Lambda<\beta<a$, fix the packing constant
from Lemma~\ref{pot:cover}, and fix
$0<\eta\leq\min\{\eta_*,1\}$. Denote by $c_0\in(0,1]$ and $c_1>0$
the constants in \eqref{eq:core-lower} and \eqref{eq:shell-lower},
respectively. All constants in the length estimates below are
independent of the step.

Start with a noncylindrical Delaunay metric with singularities at two
chosen points of $\Lambda$, transported by a conformal automorphism
of the sphere. This metric is marked nondegenerate: its radial Jacobi
fields are neutral, the first spherical modes grow at one end, and
the higher modes are coercive; see \cite[Section~II, Lemma~2.21]{MPU1996}
and \cite[Section~II]{MPU1995}. If $\Lambda$ is finite, insert its
remaining points by successive single-ball applications of the finite
gluing statement. The resulting finite-ended metric is complete,
and Lemma~\ref{pot:oldrepresentation} gives the asserted integrability
and global equation. Henceforth assume that $\Lambda$ is infinite.

Fix a stereographic pole outside $\Lambda$. All balls used in the
construction are Euclidean balls in this chart, whereas
$N_\sigma(\Lambda)$ denotes a round-metric neighborhood. Choose a
compact exhaustion $L_j\Subset\Omega$ with
$L_j\subset\operatorname{int}L_{j+1}$, and choose $\epsilon_j\in(0,1)$
so that
\begin{equation}\label{eq:product-budget}
 \sum_{j=1}^{\infty}\epsilon_j<\frac14.
\end{equation}
Thus every finite or infinite product of factors $1-\epsilon_j$ is
greater than $3/4$.

Suppose that $g_{j-1}=u_{j-1}^{2/a}g_{\mathrm{rd}}$ is a marked
nondegenerate finite-ended solution with singular set
$F_{j-1}\subset\Lambda$, and that compact sets
$W_1,\ldots,W_{j-1}\Subset\Omega$ have been constructed. Choose
$\sigma_j<2^{-j}$, decreasing with $j$, so that
$\overline{N_{\sigma_j}(\Lambda)}$ misses
$L_j\cup\bigcup_{\ell<j}W_\ell$.

For every $z\in F_{j-1}$, choose an outer radius $b_z$ such that the
closed balls $\overline{B(z,b_z)}$ are pairwise disjoint, lie in
$N_{\sigma_j}(\Lambda)$, and satisfy
\begin{equation}\label{eq:old-end-bound}
 U_{j-1}(x)\geq m_z|x-z|^{-a}
 \qquad (0<|x-z|\leq b_z)
\end{equation}
for some $m_z>0$. This follows from the positive Delaunay profile of
each end. Choose $a_z\in(0,b_z)$ and a smooth logarithmic cutoff
$f_z$, equal to one for $|x-z|\leq a_z$ and zero for
$|x-z|\geq b_z$, with
$$
 |D f_z|\leq\frac{C}{|x-z|\log(b_z/a_z)}.
$$
By decreasing $a_z$, arrange
\begin{equation}\label{eq:old-cutoff-gradient}
 |\dd f_z|_{g_{j-1}}\leq\frac{c_0^{1/a}}{\eta}
 \qquad\text{on the transition annulus}.
\end{equation}
Indeed, \eqref{eq:old-end-bound} bounds this gradient by
$C/(m_z^{1/a}\log(b_z/a_z))$.

Choose the old end sections strictly inside $B(z,a_z/2)$, and choose
the regular cap at the stereographic pole away from the logarithmic
annuli and the compact set
$$
 E_j=\Lambda\setminus\bigcup_{z\in F_{j-1}}B(z,a_z).
$$
The resulting truncated core contains these annuli and $E_j$ in its
interior. Specify $K=E_j$ before selecting the radius and coefficient
bounds in the finite gluing statement.

Since $\Hh^\beta(E_j)=0$, Lemma~\ref{pot:cover} gives finitely many
disjoint balls $B(y_i,r_i)$, with centers $y_i\in E_j$, whose
open fivefold enlargements cover $E_j$. Their tenfold enlargements
can be required to lie in the interior of the old core and in
$N_{\sigma_j}(\Lambda)$. Their maximum radius and total $\beta$-cost
can be made arbitrarily small. Equation~\eqref{pot:charge}
therefore makes $Q$ satisfy all prescribed coefficient bounds.

Apply the finite gluing statement, requiring in addition that
\begin{align}
 \|u_j-u_{j-1}\|_{C^j(L_j)}&\leq2^{-j},
       \label{eq:absolute-budget}\\
 |u_j-u_{j-1}|&\leq\epsilon_j u_{j-1}
       &&\text{on }L_j\cup\bigcup_{\ell<j}W_\ell.
       \label{eq:relative-budget}
\end{align}
These are finitely many compact approximation requirements; the
relative bounds are possible because $u_{j-1}$ has a positive minimum
on each protected compact set. The output has singular set
$F_j=F_{j-1}\cup\{y_i\}$ and remains marked nondegenerate with
noncylindrical ends. If $E_j=\varnothing$, take $u_j=u_{j-1}$.
By Lemma~\ref{pot:oldrepresentation}, every finite-step factor $u_j$
belongs to $L^p(\Sph^n)$ and satisfies the global distributional
equation.

For every new ball choose a smooth cutoff $\chi_i$, equal to one on
$B(y_i,5r_i)$ and zero outside $B(y_i,10r_i)$, with $|D\chi_i|\leq C/r_i$.
Extend all the cutoffs smoothly over the sphere, and set
\begin{equation}\label{eq:barrier-definition}
 f_j=\max\bigl(\{f_z:z\in F_{j-1}\}\cup\{\chi_i\}_i\bigr),
 \qquad W_j=\{1/4\leq f_j\leq3/4\}.
\end{equation}
The old inner balls and new open fivefold balls cover $\Lambda$.
Consequently, $f_j=1$ on an open neighborhood $V_j$ of $\Lambda$,
and both $\Sph^n\setminus V_j$ and $W_j$ are compact subsets of
$\Omega$. Since $\supp f_j\subset N_{\sigma_j}(\Lambda)$, the later
supports miss $W_j$; hence the sets $W_j$ are pairwise disjoint.

At a point of $W_j$, every new fivefold ball is avoided. If a new
cutoff $\chi_i$ is active in the maximum, its point lies in $B(y_i,10r_i)$,
and \eqref{eq:shell-lower} gives
$$
 U_j^{1/a}\geq c_1^{1/a}\frac{\eta}{r_i},
 \qquad |\dd\chi_i|_{g_j}\leq\frac{C}{\eta}.
$$
An active old cutoff lies in its logarithmic annulus, in the old core.
Every point of $W_j$ lies outside all the balls $B(y_i,5r_i)$,
and hence outside all the matching balls $B(y_i,\rho_i)$. Thus
\eqref{eq:core-lower} and \eqref{eq:old-cutoff-gradient} give
$|\dd f_z|_{g_j}\leq1/\eta$. The same estimate holds when no new
balls are inserted. Every active cutoff on $W_j$ therefore has
metric gradient bounded by $C/\eta$.

The bounds \eqref{eq:absolute-budget} imply convergence to a smooth
function $u\geq0$ locally on $\Omega$: for any fixed $m,k$, the
$C^k(L_m)$ tail is summable. Passing to the local equations yields
\begin{equation}\label{eq:local-limit-equation}
 \mathcal A_{\mathrm{rd}}u=\kappa u^p\qquad\text{on }\Omega.
\end{equation}
For $x\in L_m$, \eqref{eq:relative-budget} gives
$$
 u(x)\geq u_m(x)\prod_{j>m}(1-\epsilon_j)>0.
$$
Hence $g=u^{2/a}g_{\mathrm{rd}}$ is a smooth metric with scalar
curvature $n(n-1)$. The same product estimate on the protected sets
gives
\begin{equation}\label{eq:protected-limit}
 u\geq\frac34u_j\qquad\text{on }W_j.
\end{equation}

Each $f_j$ is the finite maximum of smooth functions and is therefore
locally Lipschitz. Along a locally absolutely continuous curve,
its derivative agrees almost everywhere with that of an active
cutoff. Together with \eqref{eq:protected-limit}, the gradient
estimates give
\begin{equation}\label{eq:curve-cutoff}
 |(f_j\circ\gamma)'(t)|
 \leq\frac{C'}{\eta}|\dot\gamma(t)|_g
 \quad\text{for almost every }t\text{ with }\gamma(t)\in W_j,
\end{equation}
where $C'$ is independent of $j$.

Let $\gamma$ be a locally absolutely continuous curve leaving
every compact subset of $\Omega$. For every sufficiently large
$j$, its starting point is outside $\supp f_j$, whereas it is
eventually in $V_j$, because $\Sph^n\setminus V_j\Subset\Omega$.
Thus it crosses $W_j$ from level $1/4$ to level $3/4$.
Integrating \eqref{eq:curve-cutoff} gives
$$
 \int_{\gamma^{-1}(W_j)}|\dot\gamma(t)|_g\,\dd t
       \ge\frac{\eta}{2C'}.
$$
The $W_j$ are pairwise disjoint, so every such curve has infinite
length. The divergent-curve criterion for Riemannian completeness
therefore proves that $g$ is complete.
Finally, $\Hh^a(\Lambda)=0$, so Proposition~\ref{prop:no-defect}
below gives the global equation and the asserted $L^p$ integrability.
\end{proof}

\subsection{Global convergence of the nonlinear sources}

The following argument is independent of the gluing construction.

\begin{proposition}\label{prop:no-defect}
Let $\Lambda\subset\Sph^n$ be closed, with $\Hh^a(\Lambda)=0$.
Suppose that $u_j>0$ belong to $L^p(\Sph^n)$, satisfy
$\mathcal A_{\mathrm{rd}}u_j=\kappa u_j^p$ in $\mathcal D'(\Sph^n)$, and converge smoothly
locally on $\Sph^n\setminus\Lambda$ to $u$. Then
$$
 u_j^p\longrightarrow u^p\quad\text{in }L^1(\Sph^n),
 \qquad
 u_j\longrightarrow u\quad\text{in }L^p(\Sph^n),
$$
and $\mathcal A_{\mathrm{rd}}u=\kappa u^p$ globally in distributions. Moreover, the source
measures $\mu_j=\kappa u_j^p\,\dd V_{g_{\mathrm{rd}}}$ satisfy
\begin{equation}\label{eq:source-morrey}
 \mu_j(B_{g_{\mathrm{rd}}}(x,r))\leq C_n r^a
 \qquad(0<r\leq r_0),
\end{equation}
where $C_n,r_0>0$ depend only on $n$. In particular, no additional
limiting source measure is supported on $\Lambda$.
\end{proposition}

\begin{proof}
Define $u$ arbitrarily on the measure-zero set $\Lambda$.

For the local estimate, we use the test-function method of
Mitidieri and Pohozaev; see \cite[Section~1, equations~(1.2)--(1.5)]{MitidieriPohozaev2004},
with numbering from their March 2004 manuscript. Put
$p'=p/(p-1)=(n+2)/4$ and choose an integer $m>2p'$.
For $B=B_{g_{\mathrm{rd}}}(x,r)$, $0<r\leq r_0\leq1$, choose
$0\leq\psi\leq1$, equal to one on $B$, supported in $2B$, with
$|\nabla\psi|\leq C_n/r$ and $|\Delta\psi|\leq C_n/r^2$.
Then $\varphi=\psi^m$ satisfies
$$
 |\mathcal A_{\mathrm{rd}}\varphi|\leq C_n r^{-2}\psi^{m-2}
             \leq C_n r^{-2}\varphi^{1/p}.
$$
Writing $X_j=\int u_j^p\varphi\,\dd V_{g_{\mathrm{rd}}}$, the equation and H\"older's
inequality give
$$
 \kappa X_j=\int u_j\mathcal A_{\mathrm{rd}}\varphi\,\dd V_{g_{\mathrm{rd}}}
 \leq C_n r^{-2}X_j^{1/p}
           \operatorname{Vol}(2B)^{1/p'}.
$$
If $X_j>0$, divide by $X_j^{1/p}$ and raise to the power $p'$;
the zero case is immediate. Since $n-2p'=a$, this yields
$$
 \int_Bu_j^p\,\dd V_{g_{\mathrm{rd}}}\leq X_j\leq C_n r^a,
$$
which proves \eqref{eq:source-morrey} in every dimension $n\geq3$.
Testing with one gives $\int u_j^p=\int u_j$; H\"older's
inequality implies $\int u_j^p\le\operatorname{Vol}(\Sph^n)$.
Fatou's lemma therefore yields $u\in L^p(\Sph^n)$.

Given $\varepsilon>0$, cover $\Lambda$ by finitely many open round
balls $B(x_i,r_i)$, with $r_i\leq r_0$ and
$\sum_i r_i^a<\varepsilon$, and let $O$ be their union.
The local estimate gives
$$
 \kappa\int_Ou_j^p\,\dd V_{g_{\mathrm{rd}}}\leq C_n\varepsilon
 \qquad\text{for every }j,
$$
and Fatou's lemma gives the same bound for $u^p$.
On $\Sph^n\setminus O\Subset\Sph^n\setminus\Lambda$, the sources
converge uniformly. Therefore
$$
 \limsup_{j\to\infty}\kappa
 \int_{\Sph^n}|u_j^p-u^p|\,\dd V_{g_{\mathrm{rd}}}\leq2C_n\varepsilon.
$$
Letting $\varepsilon\downarrow0$ proves strong $L^1$ convergence.

For $s,t\geq0$, $|s-t|^p\leq|s^p-t^p|$, since
$(t+h)^p\geq t^p+h^p$ for $h\geq0$. Thus
$$
 \|u_j-u\|_{L^p}^p\leq\|u_j^p-u^p\|_{L^1}\longrightarrow0.
$$
The factors also converge in $L^1$, so for every
$\varphi\in C^\infty(\Sph^n)$ we can pass to the limit in
$$
 \int_{\Sph^n}u_j\mathcal A_{\mathrm{rd}}\varphi\,\dd V_{g_{\mathrm{rd}}}
 =\kappa\int_{\Sph^n}u_j^p\varphi\,\dd V_{g_{\mathrm{rd}}}.
$$
This proves the global distributional equation. Strong $L^1$
convergence of the sources excludes an additional measure on
$\Lambda$.
\end{proof}

\subsection*{Declaration of competing interest}
	The authors declare that they have no known competing financial interests in this paper.
	
\subsection*{Data availability}
	No data were used for the research described in the article.
	
\subsection*{On the use of AI}
ChatGPT 6 was employed to discuss potential ideas. Kimi K3 was used for grammatical refinement, spelling correction, and suggestions on the phrasing of section and lemma headings. The authors take full responsibility for this work.

\end{document}